\documentclass[11pt,reqno,a4paper]{article}
\usepackage{geometry}
\usepackage[title]{appendix}
\usepackage{amsmath,amsthm,amstext,amscd,amssymb,euscript,url}
\usepackage{mathrsfs}
\usepackage{calrsfs}
\usepackage{epsfig}
\usepackage[inline]{enumitem}
\usepackage{todonotes}
\usepackage{mathtools}
\usepackage[absolute,overlay]{textpos}
\usepackage{graphicx}
\usepackage{subfig}
\usepackage{tikz}
\usepackage{comment}
\usetikzlibrary{matrix}
\usetikzlibrary{arrows,positioning}

\usepackage{color}
\allowdisplaybreaks

\usepackage{etoolbox}

\renewcommand{\ll}{\lambda_L}
\newcommand{\lr}{\lambda_R}

\renewcommand{\P}{\mathbb{P}}

\newcommand{\Z}{\mathbb Z}
\newcommand{\R}{\mathbb R}
\newcommand{\NN}{{\mathbf N}}

\newcommand{\N}{\mathbb N}

\newcommand{\E}{\mathbb E}

\newcommand{\Zd}{\mathbb Z^d}

\renewcommand{\phi}{\varphi}

\newcommand{\si}{\ensuremath{\sigma}}

\newcommand{\pee}{\ensuremath{\mathbb{P}}}

\newcommand{\Pres}{P^{\rm{res}}}

\newcommand{\Dcl}{D^{\text{cl}}}
\newcommand{\Dclres}{D_{\text{res}}^{\text{cl}}}
\newcommand{\Dor}{D^{\text{or}}}
\newcommand{\Dorres}{D_{\text{res}}^{\text{or}}}

\newcommand{\dd}{\mathrm{d}}
\def\1{{\mathchoice {\rm 1\mskip-4mu l} {\rm 1\mskip-4mu l}
{\rm 1\mskip-4.5mu l} {\rm 1\mskip-5mu l}}}

\newtheorem{theorem}{{\small T}{\scriptsize HEOREM}}[section]
\newtheorem{corollary}{{\bf{\small C}{\scriptsize OROLLARY}}}[section]
\newtheorem{proposition}{{\bf{\small P}{\scriptsize ROPOSITION}}}[section]
\newtheorem{lemma}{{\bf{\small L}{\scriptsize EMMA}}}[section]
\newtheorem{remark}{{\bf{\small R}{\scriptsize EMARK}}}[section]
\newtheorem{definition}{{\bf{\small D}{\scriptsize EFINITION}}}[section]
\newtheorem{induction}{{\bf{\small I}{\scriptsize NDUCTIVE HYPOTHESIS}}}[section]

\renewenvironment{proof}[1]
{\noindent{{\bf{\small{ P}{\scriptsize ROOF}}}.}\hspace{0.1cm} #1} {$\;\qed$\newline}

\newcommand{\beq}{\begin{eqnarray}}
\newcommand{\eeq}{\end{eqnarray}}

\newcommand{\ba}{\begin{align*}}
\newcommand{\ea}{\end{align*}}

\newcommand{\be}{\begin{equation}}
\newcommand{\ee}{\end{equation}}

\newcommand{\bl}{\begin{lemma}}
\newcommand{\el}{\end{lemma}}

\newcommand{\br}{\begin{remark}}
\newcommand{\er}{\end{remark}}

\newcommand{\bt}{\begin{theorem}}
\newcommand{\et}{\end{theorem}}

\newcommand{\bd}{\begin{definition}}
\newcommand{\ed}{\end{definition}}

\newcommand{\bind}{\begin{induction}}
\newcommand{\eind}{\end{induction}}

\newcommand{\bp}{\begin{proposition}}
\newcommand{\ep}{\end{proposition}}

\newcommand{\bc}{\begin{corollary}}
\newcommand{\ec}{\end{corollary}}

\newcommand{\bpr}{\begin{proof}}
\newcommand{\epr}{\end{proof}}

\newcommand{\bi}{\begin{itemize}}
\newcommand{\ei}{\end{itemize}}

\newcommand{\ben}{\begin{enumerate}}
\newcommand{\een}{\end{enumerate}}

\newcommand{\caA}{{\mathcal A}}

\newcommand{\caD}{{\EuScript D}}

\newcommand{\caE}{{\mathrsfs E}}

\newcommand{\caP}{{\mathcal P}}

\newcommand{\caX}{{\mathcal X}}

\newcommand{\nn}{\nonumber}

\newmuskip\pFqmuskip

\newcommand\pFq[6][8]{%
	\begingroup 
	\pFqmuskip=#1mu\relax
	\mathcode`\,=\string"8000
	\begingroup\lccode`\~=`\,
	\lowercase{\endgroup\let~}\pFqcomma
	{}_{#2}F_{#3}{\left[\genfrac..{0pt}{}{#4}{#5};#6\right]}%
	\endgroup
}
\newcommand{\pFqcomma}{\mskip\pFqmuskip}

\newtheorem{asu}{Assumption}
\newcommand{\basu}{\begin{asu}}
\newcommand{\easu}{\end{asu}}

\begin{document}
\title{{\bf  Boundary driven Markov gas: \\ duality and scaling limits}}

\author{G. Carinci$^{\textup{{\tiny(a)}}}$, S. Floreani$^{\textup{{\tiny(b)}}}$, C. Giardin\`{a}$^{\textup{{\tiny(a)}}}$ and F. Redig$^{\textup{{\tiny(b)}}}$\
\\
\\
\small $^{\textup{(a)}}$
\small{Universit\`a di Modena e Reggio Emilia}\\
{\small Via Campi 213/B, 41125, Modena, Italy}
\\\\
\small $^{\textup{(b)}}$
\small{Delft Institute of Applied Mathematics, Delft University of Technology}\\
{\small Van Mourik Broekmanweg 6, 2628 XE Delft, The Netherlands}
}
\maketitle

\begin{abstract}

Inspired by the recent work of Bertini and Posta \cite{bertini},  who introduced the boundary driven Brownian gas on $[0,1]$,  we study boundary driven systems of independent particles in a general setting, including particles jumping on finite graphs and diffusion processes on bounded domains in $\R^d$. We prove duality
with a dual process that is absorbed at the boundaries, thereby creating a general framework that unifies dualities for boundary driven systems in the discrete and continuum setting. 
We use duality first to show that from any initial condition the systems evolve to the  unique invariant 
measure, which is a Poisson point process with intensity the solution of a Dirichlet problem. Second, we show how the boundary driven Brownian gas arises as the diffusive scaling limit of a system of 
independent random walks coupled to reservoirs with properly rescaled intensity. 

 

 \medskip\noindent
 \emph{Acknowledgement:} S.F.\ acknowledges financial support from Netherlands Organisation for Scientific Research (NWO) through grant TOP1.17.019.

\end{abstract}
\section{Introduction}
\subsection{Background and motivation}\label{section: background}

Boundary driven systems are important in the study of non-equilibrium steady states
\cite{bertini2015macroscopic,derrida2007non, lebowitz1988microscopic}. In the context of interacting particle systems on finite graphs, boundary driving means that one adds  reservoirs at the boundaries, where particles can enter and leave the system. This is usually modeled via birth and death processes, where the birth and death rates are chosen in a manner adapted to the system. This means that the stationary measure of the reservoirs is a marginal of the stationary measure of the system. 
The simplest setting is a one-dimensional chain, where the action of the reservoirs is modeled by letting particles
enter and leave the system at left and right end. The stationary distribution of
such non-equilibrium systems and its macroscopic properties (e.g. the density profile, the current and their large deviations) are then 
the usual objects of study.

\smallskip

In the {\em discrete setting} of finite graphs, boundary driven systems of independent particles (and more generally zero-range processes) have a special status, because the non-equilibrium steady states are inhomogeneous product measures, in the case of independent particles product of Poisson distributions. For one dimensional systems, the parameters of these product measures then interpolate linearly between the  densities $\ll$ and $ \lr$ of the left reservoir and right reservoir.
For a class of  particle systems (including independent particles), one has the property of {\em duality} \cite{dmp}, which allows to express the $n$-point time-dependent correlation functions in terms of the evolution of $n$ (dual) particles. In the discrete setting, these dual particles evolve on a larger system, where absorbing
extra sites have been added, representing the action of reservoirs of the original system.
Duality  has been an essential tool to study detailed properties of different boundary driven systems such as the so-called KMP model (see \cite{KMP82}), the Exclusion process and the Inclusion process (see \cite{FGHNR}, \cite{FRSor}, \cite{franceschini2021symmetric}, \cite{frassek2021exact}).
See \cite{cggr} for an account of dualities in the discrete boundary driven setting. Given the broad applicability of duality there is the need to extend it to continuum systems.

\smallskip

In \cite{floreani} the authors started the study of self-duality beyond the discrete setting, i.e., 
self-duality of general independent Markov processes evolving as point configurations, which is the  analogue of particle configurations in the discrete setting.
There, self-duality turned out to be a  general property of the evolution of the $n$-th factorial moment measure, which can be expressed via the evolution
of $n$ (dual) particles.
The authors of \cite{floreani}  considered the setting of closed systems with a conserved number of particles.
The goal of our paper will be to initiate  the analysis of duality for {\em boundary-driven systems in the continuum}, starting from the case of independent particles.
We believe that the framework we build here can be used as well for boundary driven interacting particle systems in the continuum, but we leave this for future research. 
\smallskip

To achieve our goal, a proper definition of the action of reservoirs in the continuum has to be considered.
In the interval $[0,1]$, the naive idea  would be to study a system of
independent Brownian motions  that are absorbed at the boundaries $0$ and $1$, 
with additional creation of particles at $0$ and $1$.  
However, as was noticed in \cite{bertini}, this approach does not work, because in the continuum particles put at the boundary 
would immediately leave via that same boundary. 
Therefore, the problem of modeling reservoirs in the continuum is more involved than in the discrete setting. In \cite{bertini} the {\em boundary-driven Brownian gas} on $[0,1]$ has been defined as the sum of two independent processes: 
one process modeling the evolution of the particles initially present in the system
and moving as independent Brownian motions absorbed at $0$ and at $1$;
and another Poisson point process adding particles on $(0,1)$ with well-chosen intensity.
The creation of particles no longer takes place at the boundaries only,
rather particles are created everywhere in $(0,1)$ with an intensity
that guarantees the prescribed densities of the reservoirs. 
The authors in \cite{bertini} then proceed by proving that this process is  Markov.

\smallskip

One of the main aims in this paper is to establish in the setting of the boundary driven Brownian gas, 
the kind of duality results proved in  \cite{cggr,FRSor} for discrete boundary driven systems.
To do this, we use the set-up introduced in \cite{floreani} for closed systems in the continuum
and extend it to the boundary driven Brownian gas. 
In particular we show that the time-dependent $n$-th factorial moment measures of this system
can be written in terms of $n$ dual Brownians, absorbed at the boundaries.
Next, a second aim is to generalize this duality to the abstract setting of
general boundary driven systems of independent particles in the continuum.
For this we will need to generalize the construction of Bertini and Posta \cite{bertini}
first to systems of independent diffusion processes  evolving on regular domain $\mathfrak D\subset \R^d$ and second to systems of  general independent Markov processes
which are allowed to jump  and which thus can leave $\mathfrak D$ without hitting its boundary. As a by-product of such general construction and our duality relations two results will follow.
We shall prove that in the discrete setting of a one-dimensional chain, 
modelling the reservoirs as: i) birth and death processes
at  the boundaries or ii) by a Poissonian addition of particles
everywhere, are indeed equivalent processes.
Furthermore the boundary driven Brownian gas (in the continuum)
arises as the diffusive scaling limit of the model with birth and death processes
(in the discrete) when the intensities are also scaled with the system size.

\subsection{Duality results for Independent Random Walks}
\label{section: recap duality}
For readers convenience we recall the standard dualities of independent particles in the discrete setting, both in the case of closed and open systems.
\vskip.4cm
\paragraph{Closed systems.}
Let us consider a system of simple independent random walks, namely the Markov process $\{\eta_t, 
\; t\ge 0\}$ with $\eta_t=\{\eta_t(x)\}_{x\in \Z^d}\in\N^{\Z^d} $ where 
\begin{equation*}
\eta_t(x)\coloneqq	 \text{number of particles at}\ x \text{ at time }\ t\ge 0
\end{equation*} 
 whose generator acts on bounded and local functions $f:\N^{\Z^d}\to \R$ as 
\begin{equation}\label{eq: generator IRW}
(Lf)(\eta)=\sum_{\|x-y\|=1} \frac{1}{2}\left[\eta(x)(f(\eta+\delta_x-\delta_y)-f(\eta))+ \eta(y)(f(\eta+\delta_y-\delta_x)-f(\eta))\right].
\end{equation}
Here the sum is restricted to nearest neighbour sites and  $\eta+\delta_x-\delta_y$ denotes the configuration where a particle has been moved from $x$ to $y$ in the configuration $\eta$.
We then have that $\{\eta_t, \, t\ge 0\}$ is self-dual with self-duality function given by 
\begin{equation}\label{eq: classical dualities}
D^{\text{cl}}(\xi,\eta)=\prod_{x\in \Z^d}d(\xi(x), \eta(x))
\end{equation}
for $\xi, \, \eta\in \N^{\Z^d}$ with single-site self-duality function given by 
\begin{equation}\label{eq: classical dualities0}
d(k,n)=(n)_k \boldsymbol 1_{\{k\le n\}}:=\frac{n!}{(n-k)!} \boldsymbol 1_{\{k\le n\}}.
\end{equation}
 If we denote by $\E^{\text{IRW}}_\eta$ the expectation w.r.t. the law of the process evolving according to the generator given in \eqref{eq: generator IRW} and starting from $\eta\in \N^{\Z^d}$, the self-duality relation is then expressed in the following way: for any $\eta, \, \xi\in \N^{\Z^d}$ and for any $t\ge 0$,
 \begin{equation}\label{eq: classic sel dual relation}
 \E^{\text{IRW}}_\xi[ D^{\text{cl}}(\xi_t,\eta)]=\E^{\text{IRW}}_\eta[ D^{\text{cl}}(\xi,\eta_t)].
 \end{equation}
The self-duality functions given in \eqref{eq: classical dualities}, that we refer to as \textit{classical self-dualities} are   products of falling factorial polynomials and they have been  used to prove the hydrodynamic limit (see  \cite{dmp}). 
More recently they have been generalized to the context of systems of independent particles evolving in the continuum (e.g., $\R^d$, and more generally in Borel spaces) in \cite{floreani}.
\vskip.4cm
\paragraph{Open systems.} 
Let us further consider a system of simple independent random walks on a finite chain $ V_N:=\{ 1,\ldots, N\}$ where the boundary points $\{1,\, N\}$ are in contact with reservoirs with intensity parameters $\ll,\, \lr\in(0,\infty)$. Namely, we consider the Markov process $\{\zeta_t, \, t\ge 0\}$ with state space $\N^{V_N}$ and whose generator acts on  functions $f:\N^{V_N}\to \R$ as 
\begin{align}\label{eq: IRW with Res CGGR}
L_{\text{res}} f(\zeta)=L_{\text{bulk}}f (\zeta)+L_{\text{left}}f(\zeta)+ L_{\text{right}}f(\zeta)
\end{align}
where
$L_{\text{bulk}}$ denotes the generator of  continuous-time symmetric independent  random walkers jumping with rate $\frac{1}{2}$ over the edges $(i,i+1)$, $i\in \{1, \ldots, N-1\}$ 
and where $L_{\text{left}}, L_{\text{right}}$ denote the boundary generators, modelling the contact with the reservoirs, which are given by
\[
L_{\text{left}} f(\zeta)= \zeta(1) [f(\zeta-\delta_1) -f(\zeta)]+ \ll [f(\zeta+\delta_1)-f(\zeta)]
\]
and
\[
L_{\text{right}} f(\zeta)= \zeta(N) [f(\zeta-\delta_N) -f(\zeta)] + \lr [f(\zeta+\delta_N)-f(\zeta)].
\]
These generators describe the exit and entrance of particles via the reservoirs at left and right boundaries of the chain.
Each particle can leave the system through the right or left end at rate 1, and at rate $\ll$ (resp.
$\lr$) particles enter the system at the left (resp. right) end.
In the following we shall  call the process $\zeta_t$ the \textit{``reservoir process with parameters} $\ll, \lr$''.

In \cite{cggr} the authors proved that the reservoir process with parameters $\ll, \lr$ is dual to a system of independent random walkers
on the lattice $\{0, \ldots, N+1\}$ with absorbing boundaries.  In the dual process the absorbing sites $0$ and  $N+1$ replace the reservoirs of the original process.
With abuse of notation we shall use, for the dual process, the name 
 $\{\xi_t, \; t\ge 0\}$ as in the previous paragraph, although now, in the boundary-driven context, the dual has absorbing boundary sites.
The duality function $D^{\ll, \lr}$ can be written as
\be\label{reservoirdualfunction}
D^{\ll, \lr} (\xi, \zeta)= \ll^{\xi(0)} \lr^{\xi(N+1)} D^{\text{cl}}(\xi,\zeta)  
\ee
where $\xi\in  \N^{\{0, \ldots, N+1\}}$, $\zeta\in \N^{\{1, \ldots, N\}}$ and $D^{\text{cl}}(\cdot,\cdot)$ is given in \eqref{eq: classical dualities} but now the product is over $V_N$ and not over $\Z^d$, namely $$D^{\text{cl}}(\xi,\zeta)  =\prod_{i=1}^N d(\xi(i),\zeta(i))$$ with $d(k,n)= \frac{n!}{(n-k)!}\, \mathbf 1_{k\leq n}$.
Let us denote by $\E_\zeta^{\rm{res}}$ the expectation in the reservoir process with parameters $\ll, \lr$ starting from $\zeta \in \N^{\{1,\ldots, N\}}$. Moreover we denote by $\E_\xi^{\rm{abs}}$ the expectation in the dual process starting from an initial configuration $\xi \in \N^{\{0, \ldots, N+1\}}$.
 Then we have the following duality result: for any $\zeta\in\N^{V_N}$, $\xi\in  \N^{\{0, \ldots, N+1\}}$ and $t\ge 0$

\be\label{reservoirdual}
\E^{\rm{res}}_\zeta\left[ D^{\ll, \lr}(\xi, \zeta_t)\right] = \E^{\rm{abs}}_\xi\left[ D^{\ll, \lr}(\xi_t, \zeta)   \right] 
\ee
or equivalently 
\be\label{reservoirdual2}\E^{\rm{res}}_\zeta\left[ \ll^{\xi(0)}\lr^{\xi(N+1)}D^{\text{cl}}(\xi, \zeta_t)\right] = \E^{\rm{abs}}_\xi\left[ \ll^{\xi_t(0)}\lr^{\xi_t(N+1)}D^{\text{cl}}(\xi_t, \zeta)\right]  .
\ee
The main aim of this paper is to extend the above duality result to  general systems of boundary-driven independent particles. The random walk dynamics of each particle  will be replaced by a generic Markov process. As a consequence we shall consider  boundary driven systems of independent particles evolving not necessarily on the lattice, rather on generic regular domains $\mathfrak D\subset \R^d$, $d\ge 1$.

\subsection{Outline}\label{section: outline}
The rest of our paper is organized as follows. In Section 2 we introduce basic notations. As a preliminary step, in Section 3 we present duality results for closed systems of independent particles
in the continuum. First we recall a self-duality result from \cite{floreani}. Second, we prove a duality result, where the dual system is deterministic and follows the backward Kolmorogov equation associated to the single particle; we then use this duality result to provide a simple proof of the Doob's theorem.  Section 4 contains the main result of this paper regarding {\em boundary driven systems}.  We start by recalling the definition of the boundary driven Brownian gas on $[0,1]$, introduced in \cite{bertini}.  We then generalize this construction  to  general independent diffusion processes moving on regular domains $\mathfrak D\subset \R^d$ and   finally to general independent Markov processes which can make jumps and thus can exit $\mathfrak D$ without hitting its boundary. For those systems 
 we formulate, with increasing generality,  the duality  results  in Theorems \ref{theorem: duality BDBG}, \ref{theorem: self-dual bound driv diff} and \ref{theorem: inter a dual general gas},
 and in particular we use Theorem \ref{theorem: self-dual bound driv diff} to characterize the unique invariant 
measure of the systems.
In Section 5, we use the duality result to show that the boundary-driven Brownian gas introduced in \cite{bertini} is the scaling limit  of the reservoir process of independent random walks  with generator \eqref{eq: IRW with Res CGGR}. 
Namely, we prove that the latter equals in distribution  the 
`boundary driven random walk gas'
and that, when the parameters are scaled as $\ll/N$, $\lr/N$, it converges on the diffusive scale to the boundary driven Brownian gas with parameters $\ll$, $\lr$. 
Finally, in Section 6, orthogonal dualities are treated, extending to the continuum results from \cite{FRSor}.

\section{Setting and Notations}

We will work in the context of independent particles moving in a state space $E$, which is assumed
to be a Polish space, equipped with its Borel $\si$-algebra $\caE$.
In the relevant examples, $E=\R^d$, or $E$ is a closed subset of $\R^d$ with regular boundary, or
in the discrete setting $E=\Zd$ or a finite graph. However, for the general duality results which we state here, there
is no need to restrict to the finite dimensional setting.

\subsection{Labeled independent particles}
A single particle is moving as a Markov process
$\{X_t:t\geq 0\}$ on $E$. A finite number of (labeled) independent particles
is the process $\mathcal X_t=(X_t(1), \ldots, X_t(\NN))$ arising from joining $\NN$ independent copies of
$\{X_t:t\geq 0\}$, possibly starting
from different initial locations $X_0(i)=x_i\in E$.
We denote by $\E_{x_1,\ldots,x_{\NN}}$ the expectation of $\{\mathcal X_t, t\ge 0\}$ starting from $(x_1,\ldots,x_{\NN})$, by $S_t$ the semigroup of the Markov process $\{X_t:t\geq 0\}$, defined via $S_tf(x)= \E_x f(X_t)$, and
by $S_t^{\otimes \NN}$ the associated semigroup of $\NN$ independent copies of $\{X_t:t\geq 0\}$. By independence we have
\[
S^{\otimes \NN}_t\prod_{i=1}^\NN f_i(x_i)= \prod_{i=1}^\NN \E_{x_i} \left[ f_i(X_t(i)) \right]= \prod_{i=1}^\NN S_t f_i (x_i).
\]
We denote by $S^*_t$ the dual semigroup working on measures $\mu$ (on $(E, \caE)$), defined via
\be\label{measuresem}
\int f d S^*_t \mu=\int S_tf d\mu.
\ee
We remind the reader that we call a $\sigma$-finite measure $m$ on $E$ \textit{reversible} if 
$$\int_E S_tf \, g\, \dd m=\int_E f\, S_tg \, \dd m$$
for any $f,g\in L^2(E,m)$ and $t>0$.
Moreover, we say that the Markov process $\{X_t, \, t\ge 0\}$ is  {\em strongly reversible} if there exists a reversible $\si$-finite measure
$m$ such that the transition probability measure is absolutely continuous w.r.t.\ $m$, i.e.,
there exists a  transition density $$\mathfrak p_t: E\times E\to [0,\infty)$$ such that, for all $t>0$,
\be\label{transdens}
S_t f(x)= \int f(y) \mathfrak p_t(x,y) m(dy) = \int f(y) \mathfrak p_t(y,x) m(dy)
\ee
where the symmetry $\mathfrak p_t(x,y)= \mathfrak p_t(y,x)$ follows from the assumed reversibility of $m$.
Relevant examples to keep in mind are i) Brownian motion, where $m$ is the Lebesgue measure; ii) symmetric random walk, where
$m$ is the counting measure; iii) the Ornstein Uhlenbeck process, where $m$ is the Gaussian measure.

\subsection{Point configurations}
It is  convenient for our purposes to describe the motion of independent particles modulo permutation, i.e. via configurations. More precisely, the initial configuration associated to  $\NN$ labeled particle positions $(x_1, \ldots, x_\NN)\in E^\NN$ is defined as
\begin{equation}\label{etanot}
\eta = \sum_{i=1}^{\NN} \delta_{x_i},
\end{equation}
which is viewed as a point configuration on $E$. The configuration at time $t$ is then defined as
\be\label{etatnot}
\eta_t = \sum_{i=1}^\NN \delta_{X_t(i)}
\ee
where $X_0(i)=x_i$. Notice that by the fact that the independent particles are indistinguishable, $\{\eta_t, t\geq 0\}$
is a Markov process on the space of point configurations with total mass $\NN$.
More generally, if we have a point configuration on $E$, with potentially infinitely many particles, i.e.,
$
\eta =\sum_{i=1}^\NN\delta_{x_i}
$
where we now also allow $\NN=\infty$, then we define the configuration at time $t>0$ as in \eqref{etatnot}.
In case we work with infinitely many particles, we have to assume that the initial configuration is such that
no explosions take place, i.e., such that at any time $t>0$, the configuration
$\eta_t= \sum_{i=1}^\NN\delta_{X_t(i)}$ is a well-defined point configuration. In this paper, however, in order to avoid technicalities, we will  restrict to systems with finitely
many particles.
We denote by $\E_\eta$ the expectation in  the configuration process $\{\eta_t, t\geq 0\}$.

For a configuration $\eta$ we define its associated $n$-\textit{th factorial  measure} by
\begin{align}\label{factmom}
\eta^{(n)}:= \sideset{}{^{\neq}}\sum_{1 \leq i_1, \ldots, i_n \leq \NN} \delta_{(x_{i_1}, \ldots, x_{i_n})}
\end{align}
where the superscript $\ne$  means that the summation is over $n$ mutually distinct indexes  $i_1, \ldots, i_n$ taken from $\{1,\ldots, \NN\}$, with $\NN=\eta(E)$. The measure $\eta^{(n)}$  is a point-measure on $E^n$. Intuitively speaking, $\eta^{(n)}$ corresponds to un-normalized sampling of $n$ different particles out of the configuration $\eta$ and takes the name \textit{factorial} from the following identity: for any $B\in\mathcal E$
$$\eta^{(n)}(B^n)=(\eta(B))_n$$
with $(m)_n:=m(m-1)\cdots(m-n+1)$ denoting the $n$-th falling factorial.

An important object of study is the expectation $\E[\eta^{(n)}]$  that is called the $n$-\textit{th factorial moment measure}.  
Here  $\E$ refers to the average w.r.t.\ the randomness of the distribution of points in $\eta$. We have that $\E[\eta^{(n)}]$  is a 
measure on $E^n$, and, in particular,
$$\E[\eta^{(n)}(B^n)]=\E[(\eta(B))_n]$$
provides the $n$-th factorial moment of the number of points of $B\in \mathcal E$.
An important special case is when the points in $\eta$ are distributed according to a Poisson point process
with intensity measure $\lambda$: it is well know (see, e.g.,  \cite[(4.11)]{last} ) that in this case one has
\be\label{factmom}
\E[\eta^{(n)}]= \lambda^{\otimes n}
\ee
which  is a particular instance of the Mecke's equation. 

In the next sections we will study, by duality,  the expectation of the $n$-th factorial measure of the configuration at time $t$, i.e. $\E_\eta[\eta_t^{(n)}]$, which will be called \textit{the $n$-th factorial moment measure at time} $t$.

\section{General duality results for independent particles}
In this section we review some known duality results for closed (i.e., without reservoirs) systems of independent particles: namely self-duality and duality w.r.t. deterministic systems.

\subsection{Intertwining  and self-duality}\label{section: self inter and self dual without res}
We now recall an intertwining and a self-duality result  for independent particles taken from \cite{floreani}. As already mentioned, in order
to avoid technicalities,  the  results below are stated for finitely many particles. However, whenever the infinitely many particle limit is well-defined, by passing to this limit, the result extends immediately to the infinite case.

The following result (originally stated in  \cite[ Theorem 3.1 and 3.2]{floreani}) states that the expectation of the $n$-th factorial moment measure at time $t$ can be expressed in terms
of $n$ independent evolutions. 

\bt\label{floreanithm}
Let $\eta$ be a finite point configuration as defined in \eqref{etanot}. Assume that the particles evolve independently
according to the Markov process $\{X_t:t\geq 0\}$.
\begin{itemize}
\item[a)] (Intertwining)  The following identity holds
\be\label{duality}
\E_\eta[\eta^{(n)}_t] = (S_t^{\otimes n})^* \eta^{(n)},
\ee
where $S_t^*$ is the dual semigroup defined in \eqref{measuresem}.
\item[b)] (Self-duality) If $\{X_t:t\geq 0\}$ is strongly reversible with reversible measure $m$ then
one can express the density of $n$-th factorial moment measure $\E_\eta[\eta^{(n)}_t] $
w.r.t.\ $m^{\otimes n}$ via
\be\label{strongrevfactmom}
\frac{\dd\E_\eta[\eta^{(n)}_t] }{\dd m^{\otimes n}} (z_1,\ldots, z_n)
=
\int \prod_{i=1}^n \mathfrak p_t(z_i, y_i) \eta^{(n)}(\dd(y_1, \ldots, y_n)),
\ee
where $\mathfrak p_t(\cdot, \cdot)$ is the transition density defined in \eqref{transdens}.
\end{itemize}
\et
We provide here a  proof of the above result, which relies on generating functions. This generating function approach is well suited to study boundary driven systems in Section \ref{section: boundr driv syst} below.

\bpr
We start from the following identity from \cite[Lemma 4.11]{last},  for a general finite random point configuration. Let $u: E\to (0,1)$ then
\be\label{last}
\exp\left(\int \log(1-u(z))\eta(\dd z)\right)= 1+\sum_{n=1}^\infty\frac{(-1)^n}{n!}\int u^{\otimes n}(z_1,\ldots, z_n)
\eta^{(n)} (\dd(z_1, \ldots, z_n)).
\ee
We can now use this identity to prove \eqref{duality}. Let us adopt the abbreviation  $u_t(z)= S_tu(z)=\E_z[u(X_t)]$. Using the independence of the processes $X_t(i)$, $i\in\{1,\ldots, N\}$, we compute
\begin{align}\label{bonk}
&\nonumber \E_\eta\left[\exp\left(\int \log(1-u(z))\eta_t(\dd z)\right)\right]
=\E_{x_1,\ldots,x_{\NN}}\left[\prod_i (1-u(X_t(i)))\right]\nonumber\\
 &= \prod_i \E_{x_i}\left[1-u(X_t(i))\right]=\prod_i (1-u_t(x_i))
 \nonumber=
 \exp\left(\int \log(1-u_t(z))\eta(\dd z)\right)
 \nonumber\\
 &=
1+\sum_{n=1}^\infty\frac{(-1)^n}{n!}\int u_t^{\otimes n}(z_1,\ldots, z_n) \eta^{(n)} (\dd(z_1, \ldots, z_n) )
\nonumber\\
&=
1+\sum_{n=1}^\infty\frac{(-1)^n}{n!}\int u^{\otimes n}(z_1,\ldots, z_n) (S^{\otimes n}_t)^*\eta^{(n)} (\dd(z_1, \ldots, z_n)),
\end{align}
where we used \eqref{last} in the fourth identity.
On the other hand, using \eqref{last} once more, we have
\be\label{bang}
\E_\eta\left[\exp\left(\int \log(1-u(z))\eta_t(\dd z)\right) \right]= 1+\sum_{n=1}^\infty\frac{(-1)^n}{n!}\int u^{\otimes n}(z_1,\ldots, z_n)
\E_\eta[\eta^{(n)}_t] (\dd(z_1,\ldots,z_n))
\ee
and therefore, from \eqref{bang} and \eqref{bonk} we conclude
\beq\label{series}
&&1+\sum_{n=1}^\infty\frac{(-1)^n}{n!}\int u^{\otimes n}(z_1,\ldots, z_n)
\E_\eta[\eta^{(n)}_t] (\dd(z_1,\ldots,z_n))
\nonumber\\
&=&
1+\sum_{n=1}^\infty\frac{(-1)^n}{n!}\int u^{\otimes n}(z_1,\ldots, z_n) (S^{\otimes n}_t)^*\eta^{(n)} (\dd(z_1, \ldots, z_n)).
\eeq
Because this holds for all $u$, identifying term by term in the above series and using a standard density argument for symmetric functions (linear combinations
of functions of the form $u(z_1)u(z_2)\ldots u(z_n)$ are dense in the set of symmetric functions), we obtain \eqref{duality}.

If in addition we assume strong reversibility, we then have, for any $f:E^n\to \R$ {bounded},
\begin{align*}
&\int f(z_1,\ldots,z_n) \E_{\eta}[\eta^{(n)}_t](\dd(z_1,\ldots,z_n))=\int f(z_1,\ldots,z_n) (S^{\otimes n}_t)^* \eta^{(n)}(\dd(z_1,\ldots,z_n))\\
&=\int (S^{\otimes n}_t f)(z_1,\ldots,z_n)  \eta^{(n)}(\dd(z_1,\ldots,z_n))\\
&=\int \left(\int f(y_1, \ldots, y_n)\prod_{i=1}^n\mathfrak p_t( z_i,y_i)m^{\otimes ^n}(\dd (y_1,\ldots,y_n))\right) \eta^{(n)}(\dd (z_1,\ldots,z_n))\\
&=\int \left(\int f(z_1, \ldots, z_n)  \prod_{i=1}^n\mathfrak p_t( z_i,y_i)  \eta^{(n)}(\dd (y_1,\ldots,y_n))\right)   m^{\otimes ^n}(\dd (z_1,\ldots,z_n)),
\end{align*}
from which \eqref{strongrevfactmom} follows.
\epr

\br 
{\em
As noticed in \cite[Remark 2.3(iii)]{floreani}, for the system of independent random walks on $\Z^d$ with generator given in \eqref{eq: generator IRW}, \eqref{strongrevfactmom} is equivalent to the classic self-duality relation \eqref{eq: classic sel dual relation}. Indeed for a singleton $(z_1,\ldots,z_n)$, $z_i\in\Z^d$, we have the relation (see \cite[Lemma 2.1]{floreani})
\be \label{eq: relation fact meas clas dual}
\eta^{(n)}(\{(z_1, \ldots, z_n)\})= D^{\text{cl}}(\delta_{z_1}+\ldots + \delta_{z_n}, \eta)
\ee
with $D^{\rm{cl}}$ defined in \eqref{eq: classical dualities} and thus $$
\E_\eta[\eta^{(n)}_t](\{(z_1,\ldots,z_n)\})= \E^{\rm{IRW}}_\eta[ D^{\text{cl}}(\delta_{z_1}+ \ldots +\delta_{z_n}, \eta_t)]$$ and $$\int \prod_{i=1}^n \mathfrak p_t(z_i, y_i) \;\eta^{(n)}(\dd(y_1, \ldots, y_n))\nn = \E^{\rm{IRW}}_{\delta_{z_1}+\ldots+\delta_{z_n}}[ D^{\rm{cl}}(\xi_t, \eta)],$$
where $\xi_t$ denotes the configuration of independent random walks at time $t$ starting from $\xi_0=\delta_{z_1}+\ldots+\delta_{z_n}$.
}
\er

\subsection{Duality w.r.t. the associated deterministic system}\label{section: duality with deterministic system} 
The so-called ``associated deterministic system'' is a dynamical system on functions $f:E\to\R$ which follows the flow of the Kolmogorov backwards equation of the Markov process $\{X_t, \, t\ge 0\}$.
More precisely for $f:E\to\R$ we define $f_t(x)=S_t f(x)=\E_x[f(X_t)]$. This flow $f_t$ is the solution of the system of ODE given by
\be\label{detpr}
\frac{df_t(x)}{dt} =\mathcal L f_t(x),
\ee
with $\mathcal L$ being the Markov generator associated to $\{S_t, t\ge 0\}$. Notice that, by the Markov semigroup property, $f_t>0$ when $f>0$. For $f:E\to(0,\infty)$ and a labeled process $\{  \mathcal X_t, t\ge 0  \}=\{  (X_t(1),\ldots,X_t(\NN)) ,t\ge0\}$ initialised from $\mathcal X = (x_1,\ldots,x_{\NN})$, we define the function 

\be\label{det}
\mathcal D(f,\mathcal X_t)=\prod_{i=1}^{\NN} f(X_t(i))
\ee
or, alternatively, in terms of the point configuration process $\{\eta_t, t\ge0\}$ we set 
$$D(f,\eta_t):=e^{\int \log f \dd\eta_t}.$$
Duality between the configuration process and the deterministic system is  then formulated as follows.
\bt
{
The process $\{\caX_t: \, t\geq 0\}$ is dual to the deterministic evolution on functions $f: E\to(0,\infty)$ defined via
$f_t(x)= S_t f(x)=\E_x [f(X_t)]$, with duality function
$\caD(f, \caX)= \prod_{i=1}^{\NN} f(X(i))$}, i.e., 
\be\label{dtxdualgen}
\E_{\caX} \left[\caD(f, \caX_t)\right]= \caD(f_t, \caX),
\ee
or, equivalently, in terms of the point configuration process 
\be\label{colanko}
\E_{\eta}\left[ D(f,\eta_t)\right]=D(f_t,\eta).
\ee
\et
\begin{proof}
	The proof is straightforward, indeed by the independence of the particles and by the definition of $f_t$, we have
	$$\E_{\caX}\left[\prod_{i=1}^{\NN} f(X_t(i))\right]= \prod_{i=1}^{\NN} \E_{\caX} \left[ f(X_t(i))\right]= \prod_{i=1}^{\NN} f_t (x_i).
$$
\end{proof}

 \subsubsection{Doob's theorem}\label{section: Doob theorem}

Let us now consider the connection between the duality result of Section \ref{section: duality with deterministic system} with the time evolution of Poisson point processes.
It is well-known that independent Markovian particle evolutions preserve Poisson processes: 
we refer to this result as Doob's theorem but it can also be viewed as a consequence of the random displacement theorem (see, e.g., \cite{last}).

We briefly recall the definition of a Poisson point process.
For a function $\rho: E\to [0,\infty)$ and a $\sigma$-finite measure $m$ on $(E,\mathcal E)$ the Poisson point process with
intensity measure $\rho(z) m(dz)$ is defined
as the random point configuration $\eta= \sum_{i=1}^{\NN} \delta_{x_i}$, defined on a probability space $(\Omega, \caA, \pee)$ such that
\ben
\item For every $\omega\in \Omega$, the map $\mathcal E \ni A\to \eta(\omega, A)$ is a $\N$-valued measure on the $\si$-algebra $\caE$.
\item For $A_1, \ldots A_n\in  \mathcal E$, $n$ disjoint measurable subsets of $E$, $\{\eta(A_i), i=1,\ldots, n\}$ are independent
Poisson random variables with parameter $m_i=\int_{A_i} \rho(z)m(\dd z)$.
\een
See \cite{last} for background on Poisson point processes.
We denote by $\caP_\rho$ the probability in  the Poisson point process with intensity measure $\rho(z) m(\dd z)$ on the space of point configurations.

We recall the reader that a Poisson point process is uniquely characterized by its Laplace functional, i.e.,
by
\be\label{poislap}
\int \left(e^{\int_E f(z) \eta(\dd z)}\right) \caP_\rho (\text{d}\eta)= e^{\int (e^{f(z)}-1)\rho(z) m(\dd z)}
\ee
for all $f$ for which the integral $\int (e^{f(z)}-1)\rho(z) m(\dd z)$ is finite.

We denote by $\E_{\caP_\rho}$ the expectation of the process  of independent particles moving according to the Markovian dynamics corresponding to the semigroup $S_t$ whose associated point configuration is initially distributed as $\caP_\rho$.
The following result then proves Doob's theorem via the duality \eqref{dtxdualgen}.
\bt\label{doob}
Let $\{  \mathcal X_t, t\ge 0  \}=\{  (X_t(1),\ldots,X_t(\NN)) ,t\ge0\}$ be a system of independent particles initialized  at time zero from a Poisson point configuration
with intensity measure $\rho(z) m(\dd z)$, where $m$ is a reversible measure of the Markov process $\{X_t:\, t\ge 0\}$.
Then the distribution of the $\NN$ particles at time $t\ge 0$, namely
the random point configuration $\sum_i \delta_{X_t(i)}$,  is a Poisson point configuration with
intensity measure $\rho_t(z)m(dz)$, where $$\rho_t(z)= \E_z [\rho (X_t)]=S_t\rho (z).$$

\noindent
More generally, if $m$ is a stationary measure of $\{X_t:\, t\ge 0\}$, the  Poisson point process is mapped to a Poisson point process with intensity measure $$\rho_t=S_t^* \rho,$$ where $S_t^*$ denotes the adjoint semigroup of $S_t$.
\et
\bpr
Using \eqref{colanko} and \eqref{poislap}, we obtain
\beq
\int e^{\int \log f(z) \eta_t(\dd z )}\caP_\rho (\text{d}\eta)&=& \int D(f, \eta_t)\caP_\rho (\text{d}\eta)
\nonumber\\
&=& \int e^{\int \log f_t(z)\eta(\dd z)}\caP_\rho (\text{d}\eta)
\nonumber\\
&=&
e^{ \langle (f_t-1), \rho\rangle_{L^2(m)}}
\nonumber\\
&=& e^{ \langle (f-1), S^*_t \rho\rangle_{L^2(m)}}
\eeq
From this we infer that $\eta_t$ is again a Poisson point process
with intensity $\rho_t(z) m(\dd z)$ where $\rho_t(z)= \E_z[ \rho(X(t))]$ if $S_t$ is self adjoint in $L^2(m)$ and $\rho_t(z)=S^*_t \rho(z)$ in the general case.
\epr
\bc
In the setting of Theorem \ref{doob}, the Poisson point processes with intensity measure $\rho\cdot m(dz)$ parametrized by a constant $\rho>0$ are reversible for $\{  \mathcal X_t, t\ge 0  \}$.
More generally, if $m$ is a stationary measure of $\{X_t:\, t\ge 0\}$, the Poisson point process is stationary if and only if $$S_t^* \rho=\rho.$$
\ec
\bpr
When $\rho$ is constant we have, using \eqref{colanko}  and \eqref{poislap},
\allowdisplaybreaks\begin{align*}
&\E_{\caP_\rho} \left(\E_\eta\left[ e^{  \int \log f\dd \eta_t }\right]e^{\int \log g \dd \eta}\right)=\E_{\caP_\rho} \left(e^{ \int \log S_tf\dd \eta }e^{\int \log g \dd \eta}\right)\\
&=\E_{\caP_\rho} \left(e^{ \int \log\left( (S_tf)g\right)\dd \eta }\right)= e^{\rho\int ((S_t f) g-1)\dd m}
\end{align*}

\noindent
and using the self-adjointness of $S_t$ we obtain
$$ e^{\rho\int ((S_t f) g-1)\dd m}= e^{\rho\int ((S_t g) f-1)\dd m}=\E_{\caP_\rho} \left(\E_\eta\left[  e^{ \int \log g\dd \eta_t }\right]e^{\int \log f \dd \eta}\right) $$
which implies reversibility of $\caP_\rho$.

The second statement follows immediately from Theorem \ref{doob}.
\epr

\section{Duality for boundary driven systems of independent particles}\label{section: boundr driv syst}
In this section we will present a duality result for boundary driven systems of independent particles which generalizes previous results of that type obtained only in the discrete setting \cite{cggr}, \cite{FRSor}. In Section \ref{bertsec}, we recall the  definition of the \textit{boundary driven Brownian gas} recently introduced in \cite{bertini} and we state a duality result in this context (see Theorem \ref{theorem: duality BDBG} below). In Section \ref{section: boundary driven independent particles} we consider more general systems  of independent diffusion processes on regular domains $\mathfrak D\subset \R^d$. We prove  first an intertwining result and secondly a duality result under an extra assumption on the transition probabilities of a single particle. In Section \ref{section: general boundary driven gas} we introduce a further generalization of the construction of Bertini and Posta in \cite{bertini}, namely  boundary driven  Markov processes with jumps, which  can exit the  domain without hitting its boundary. 
\subsection{The boundary driven Brownian gas on $[0,1]$: definition  and duality}\label{bertsec}
Let $E=[0,1]$ and denote by $\{W_t, t\ge 0\}$ a standard Brownian motion absorbed upon hitting $0$ or $1$. Let us denote by $\tau_0, \tau_1$ the hitting
times of $0$, resp.\ $1$, of $\{W_t,  t\geq 0\}$. We denote by $\pee^{\rm{abs}}_x$ and by $S_t$ respectively the distribution of the trajectories of $\{W_t, t\ge 0\}$ starting from $x\in [0,1]$ and the semigroup of the process. It is well known that the transition probability  $p_t(\cdot,\cdot):[0,1]\times \mathcal B([0,1])\to[0,1]$ of the absorbed Brownian motion satisfies 
 \begin{equation}\label{eq: absolutely continuity pt}
p_t(x, \dd y)= \mathfrak p_t(x, y)\, \dd y \qquad \forall x,y\in (0,1)
\end{equation}
with $\mathfrak p_t(x,y)=\mathfrak p_t(y,x)$ a symmetric function referred as transition density (see, e.g., \cite[p. 122]{handbook} for an explicit formula of $\mathfrak p_t(x,y)$). With a slight abuse of notation we denote by $\mathfrak p_t(x,0)$ (respectively $\mathfrak p_t(x,1)$) the  probability,  starting from $x\in [0,1]$, of being absorbed at $0$ (resp. at $1$) by the time $t\ge 0$. We then have, for any $x\in (0,1)$,
 \begin{equation}\int_0^1 \mathfrak p_t(x,y) \, \dd y + \mathfrak p_t(x,0)+ \mathfrak p_t(x,1)=1
\end{equation}
and for any $f:[0,1]\to \R$ {bounded }
$$S_tf(x)=\int_0^1 \mathfrak p_t(x,y)f(y)\, \dd y+ f(0)\mathfrak p_t(x,0)+ f(1)\mathfrak p_t(x,1).$$
\noindent
For $\xi:=\sum_{i=1}^\NN\delta_{x_i}$, $x_i\in (0,1)$ and $\NN\in \N$, we then consider the point configuration (on $[0,1]$) valued  Markov process given by 
$$\begin{cases}\xi_t:=\sum_{i=1}^\NN\delta_{W_{t}(i)}\ ,&\\
\xi_0=\xi\end{cases}$$
where $\{W_t(i)\}_{t\ge 0}$ are independent copies of $\{W_t\}_{t\ge 0}$ such that $W_0(i)=x_i$ for any $i\in[\NN]$. The transition function $P_t(\xi,\cdot)$ of the process $\{\xi_t, t\ge 0\}$ is then given by the image of $\otimes_{i=1}^\NN  p_t(x_i,\cdot)$ under the mapping $(x_i)_{i=1}^\NN\to \sum_{i=1}^\NN\delta_{x_i}.$ For $\boldsymbol x=(x_1,\ldots,x_\NN)\in (0,1)^\NN$, we denote by $\E^{\rm{abs}}_{\boldsymbol x}$ the expectation in
the process $\{\xi_t, t\ge 0\}$ starting from $\xi_0=\sum_{i=0}^N\delta_{x_i}$. 
Finally, let $\Theta_{t}$ be a Poisson point configuration on $(0,1)$ with time dependent intensity $\lambda_t(dx)$ given by 
\begin{equation}\label{intens}
\lambda_t(\dd x)=\lambda(t,x) \, \dd x
\end{equation}
and
\begin{align}
\nonumber \lambda(t,x)&=\ll \P^{\rm{abs}}_x(\tau_0\le t)+\lr\P^{\rm{abs}}_x(\tau_1\le t)\\
&=\ll \mathfrak p_t(x,0)+\lr\mathfrak p_t(x,1)
\end{align}
for some $\lambda=(\ll, \lr)\in\R^2_+$. 
Moreover $\{\xi_t\}_{t\ge 0}$ and $\{\Theta_t\}_{t\ge 0}$ are independent. The process $\{\Theta_t\}_{t\ge 0}$, by adding particles in the bulk $(0,1)$,  models in turn the effect of the reservoirs  at $0$ and $1$ (cf. \cite[(2.1), (2.2)]{bertini}).
The \textit{boundary driven Brownian gas}  is then defined, for any $t> 0$, by 
\begin{equation}\label{ka}\eta_t=\xi_t {\big |}_{(0,1)} +\Theta_t
\end{equation}
viewed as a point configuration on $(0,1)$ and such that $\eta_0=\xi_0 {\big |}_{(0,1)}$. Here $\xi_t {\big |}_{(0,1)}$ denotes the restriction of the point configuration $\xi_t$ to $(0,1)$.

 The motivation for this definition can be found in \cite{bertini}. In Section \ref{section: scaling limit} below we will show how the boundary driven Brownian gas arises as a scaling limit of the reservoir process on a chain $\{1,\ldots, N\}$ defined in \eqref{eq: IRW with Res CGGR}.
\vskip.3cm
\noindent
Let us recall that $\eta^{(n)}$ denotes the $n$-th  factorial  measure
corresponding to the initial configuration $\eta_0=\eta$ made of $\NN$ particles, and $\eta_t^{(n)}$ denotes the $n$-th factorial
 measure corresponding to $\eta_t$, i.e. the configuration at time $t$ with $\NN_t$ particles. We denote by $\E^\lambda_\eta$ the expectation in the process
defined via \eqref{ka} initialized from $\eta$.  We will use the following abbbreviations: $\boldsymbol x=(x_1,\ldots,x_n)$, for $I=(i_1,\ldots,i_k)\subset \{1, \ldots, n\}$ we put $\boldsymbol x_{I}=(x_{i_1},\ldots, x_{i_k})$ and we write $[k]$ for $\{1,\ldots,k\}$. We shall also use the following shorthand for the transition density in the rest of the paper
\be
\mathfrak p^{(n)}_t(\boldsymbol x,\boldsymbol y)= \prod_{i=1}^n \mathfrak p_t(x_i,y_i).
\ee
\vskip.2cm

For the boundary driven Brownian gas the following duality result holds, where the dual process is  a system of independent absorbed Brownian motions.

\begin{theorem}\label{theorem: duality BDBG}
For the boundary driven Brownian gas $\{\eta_t, t\ge 0\}$, the $n$-th factorial moment measure at time $t>0$ is absolutely continuous w.r.t.  $m^{\otimes n}$, with $m$ denoting the Lebesgue measure on $(0,1)$ with the following density:

\be\label{eq: duality bound driv brownian gas}
\frac{\dd\E^{\lambda}_\eta[\eta_t^{(n)} ]}{\dd m^{\otimes n}} (\boldsymbol z)=\sum_{I\subset [n]}   \E^{\rm{abs}}_{\boldsymbol z_{I}}\left[ \ll^{\xi_t(\{0\})}\lr^{\xi_t(\{1\})}\boldsymbol 1_{\{\xi_t(\{0,1\})=|I|\}} \right] \int_{(0,1)^{n-|I|}}\mathfrak p_t^{(n-|I|)}(\boldsymbol z_{[n]\setminus I},\boldsymbol y)     \eta^{(n-|I|)} (\dd\boldsymbol y) .      
\ee
\end{theorem}

This result can be read as a duality relation in the spirit of \eqref{strongrevfactmom}: in order to know the $n$-th order factorial moment
measure at time $t>0$, one has to follow (not more than) $n$ dual particles. However, due to the presence
of reservoirs, we have factors 
$$\E^{\rm{abs}}_{\boldsymbol z_{I}}\left[ \ll^{\xi_t(\{0\})}\lr^{\xi_t(\{1\})} \boldsymbol 1_{\xi_t(\{0,1\})=|I|} \right] $$ which can be considered as corresponding
to $|I|$ ``absorbed'' dual particles. This result has to be compared with the discrete setting, namely \eqref{reservoirdual2}, where an analogous term multiplying the product of falling factorial polynomials appears in the duality function and  the process with reservoirs
is dual to an absorbing  process with two extra sites associated to the reservoirs (see
\cite{cggr}, \cite{FRSor} and \cite{FGHNR}).

In the next subsection we state and prove a more general version of Theorem \ref{theorem: duality BDBG},
which applies to independent diffusions on regular domains $\mathfrak D\subset \R^d$ and includes also
an intertwining result.  \eqref{eq: duality bound driv brownian gas} for the boundary driven Brownian gas on $(0,1)$ will then follow as a particular case of Theorem \ref{theorem: self-dual bound driv diff}.

\subsection{Boundary driven diffusion processes: definition and duality}\label{section: boundary driven independent particles}

Let $\mathfrak D$ be a regular domain of $\R^d$, where by  regular domain we mean an open, simply connected and bounded subset  $\mathfrak D\subset \R^d$ such that its boundary $\partial  \mathfrak D$ is Lipschitz. Let  $\{Y_t, t\ge 0\}$ be the diffusion process on $\R^d$ with  generator
\be
\label{gen-diffus}
\mathcal L=\sum_{i,j=1}^d \frac{\partial}{\partial x_i}\left(  a_{i,j}(x) \frac{\partial}{\partial{x_j}}\right)+ \sum_{i=1}^d
b_i(x)\frac{\partial}{\partial{x_i}}
\ee
with regular coefficient functions $a_{i,j}$, $b_j$  and with $a=(a_{i,j})$ symmetric, non-degenerate and positive definite. We then  denote by $\{X_t, t\ge 0\}$ the Markov process on $\bar{\mathfrak D}=\mathfrak D \cup \partial \mathfrak D$ with semigroup $\{S_t, t\ge 0\}$, evolving as $\{Y_t, t\ge 0\}$ on $\mathfrak D$ and absorbed upon hitting $\partial D$.  More specifically the regularity assumptions on the coefficients are that  $\frac{\partial^2 a_{i,j}}{\partial_{x_i}\partial_{x_\ell}}$ and $\frac{\partial b_i}{\partial_{x_j}  }$ are locally uniformly Hölder continuous on $\mathfrak D\cup \partial \mathfrak D$ (see, e.g., \cite{k1978}). Denote by  $\P^{\rm{abs}}_x$ (resp. $\E^{\rm{abs}}_x$) the distribution (resp. the expectation) of the trajectories of  $\{X_t, t\ge 0\}$ starting from $x\in \mathfrak D$. We then assume that 
\begin{align}
\P^{\rm{abs}}_x\left ( \tau_{\partial \mathfrak D}<\infty\right)=1, \quad \forall x \in \mathfrak D
\end{align}
where $\tau_{\partial \mathfrak D}$ denotes the hitting time of $\partial \mathfrak D$.

For $\xi:=\sum_{i=1}^\NN\delta_{x_i}$, $x_i\in\mathfrak D$, we consider the point configuration (on $\bar{\mathfrak D}$) valued Markov process given by 
$$\begin{cases}\xi_t:=\sum_{i=1}^\NN\delta_{X_{t}(i)}\ ,&\\
\xi_0=\xi\end{cases}$$
where $\{X_t(i)\}_{t\ge 0}$ are independent copies of $\{X_t\}_{t\ge 0}$ such that $X_0(i)=x_i$ for any $i\in[\NN]$. For $\boldsymbol x=(x_1,\ldots,x_\NN)\in \mathfrak D^\NN$, we denote by $\E^{\rm{abs}}_{\boldsymbol x}$ the expectation in
the process $\{\xi_t, t\ge 0\}$ starting from $\xi_0=\sum_{i=0}^N\delta_{x_i}$.

Let now $\lambda:\partial  \mathfrak D\to \R_+$ be a bounded measurable function giving the reservoir intensity at any $x\in\partial  \mathfrak D$.  If $\lambda$ satisfies the just mentioned assumptions it is said to be regular. Finally we define $\Theta_t$ the Poisson point process on $ \mathfrak D$ with time dependent intensity $\lambda_t(\dd x)$ given by 
\begin{equation}\label{eq: lambda t multidim BM}
\lambda_t (\dd x)=\left(\int_{\partial  \mathfrak D} \lambda (y) \, \P_x( \tau_{\partial \mathfrak D}\le t, X_{\tau_{\partial \mathfrak D}} \in\dd y )\right) \mu(\dd x),
\end{equation}
for a finite measure $\mu$ on $\mathfrak D$.
The \textit{boundary driven diffusion gas in the domain} $\mathfrak D$ \textit{with reservoir intensity} $\lambda$ \textit{and a priori measure} $\mu$, denoted by $\{\eta_t, t\ge 0\},$ is then given, for any $t\ge 0$, by 
\be\label{MDG}
\begin{cases}\eta_t=\xi_t{\big |}_{\mathfrak D} +\Theta_t\ ,&\\
	\eta_0=\xi_0=\sum_{i=1}^N \delta_{x_i}, \ x_i\in \mathfrak D \end{cases}
\ee
viewed as a point configuration on $\mathfrak D$, where $\xi_t{\big |}_{\mathfrak D}$ is the restriction of $\xi_t$ to $\mathfrak D$.

We denote by $\E^\lambda_\eta$ the expectation in the process
defined via \eqref{MDG} initialized from $\eta$. 
Following the strategy of \cite{bertini}, it can be shown that $\{\eta_t,\, t\ge 0\}$ is a  Markov process when the transition probability $p_t(x,\dd y)$ of the process $\{X_t, t\ge 0\}$ satisfies 
\be\label{eq: condition for Markov}
p_t(x,\dd y) m(\dd x)=p_t(y,\dd x) m(\dd y) \qquad \text{on} \qquad  \mathfrak D \times \mathfrak D
\ee 
for a finite measure $m$ and for the reservoir  intensity in \eqref{eq: lambda t multidim BM} we choose $\mu=m$. We refer to  Theorem \ref{Markovianity} in the Appendix for further details.

We are now ready to state the main results of this section, namely 
a general intertwining  relation for the factorial moment measure at time $t>0$ of the  boundary driven diffusion processes on a $d$-dimensional regular domain $\mathfrak D$ and a duality result, under an extra symmetry  assumption on the transition probability of $\{X_t, t\ge 0\}$ (see \eqref{eq: cond for duality} below), generalizing Theorem \ref{theorem: duality BDBG}.

\begin{theorem}\label{theorem: self-dual bound driv diff}
	\begin{itemize}
			Let $\{\eta_t, t\ge 0\}$ be the boundary driven diffusion gas  defined in \eqref{MDG}. Then for all $n\in\N$ and $t\ge 0$, it holds:
\item[a)] for all bounded, measurable  and permutation invariant  $f:\mathfrak D^n \to \R$ 
	\begin{equation}\label{eq: inter mult dim diff proc}
	\E^\lambda_\eta\left[\int_{\mathfrak D^n} f(\boldsymbol z)\eta^{(n)}_t (\dd \boldsymbol z)  \right] = \sum_{k=0}^n {n\choose k} \int_{\mathfrak D^n} f(\boldsymbol z) \lambda_t^{\otimes k}(\dd\boldsymbol z_{[k]})
	\otimes (S^{\otimes n-k}_t)^*\eta^{(n-k)} (\dd \boldsymbol z_{[n]\setminus [k]});
	\end{equation}
\item[b)] assume further  that the transition probability of $\{X_t,t\ge 0\}$ satisfies 
\begin{equation}\label{eq: cond for duality}
p_t(x,\dd y)=\mathfrak p_t(x,y) m(\dd y)
\end{equation} for a symmetric function $\mathfrak p_t(x,y)$ and a finite measure $m$ on $\mathfrak D$. Then, choosing $\mu=m$, the following holds
\begin{align}\label{eq: duality mult dim diff}
	&\frac{\dd\E^\lambda_\eta[\eta_t^{(n)} ]}{\dd m^{\otimes n }} (\boldsymbol z)=\sum_{I \subset [n]}  \E^{\rm{abs}}_{\boldsymbol z_{I}}\left[e^{\int_{\partial\mathfrak D} \log (\lambda)\dd\xi_t}\boldsymbol 1_{\{\xi_t(\partial \mathfrak D)=|I|\}} \right]	\int_{\mathfrak D^{n-|I|}}\mathfrak p_t^{(n-|I|)}(\boldsymbol z_{[n]\setminus I},\boldsymbol y)     \eta^{(n-|I|)} (\dd\boldsymbol y)
	\end{align}
	
	\end{itemize}
\end{theorem}

\allowdisplaybreaks\begin{remark}
\begin{itemize}
	\item[i)]	Notice that, if $\lambda(x)\in\{\ll,\lr\}$ for any $x\in \partial \mathfrak D$ and it is regular (as defined above), setting $\partial \mathfrak D_L=\{x\in \partial \mathfrak D: \ \lambda(x)=\ll \}$, we then have  
		\begin{multline}
		\frac{\dd\E^\lambda_\eta[\eta^{(n)} ]}{\dd m^{\otimes n}} (\boldsymbol z)
		=\sum_{I\subset [n]}    \E^{\rm{abs}}_{\boldsymbol z_{I}}\left[ \ll^{\xi_t(\partial \mathfrak D_L)}\lr^{\xi_t(\partial \mathfrak D\setminus\partial \mathfrak D_L)}\boldsymbol 1_{\{\xi_t(\partial E)=|I|\}} \right]
		\\\times\int_{\mathfrak D^{n-|I|}}\mathfrak p_t^{(n-|I|)}(\boldsymbol z_{[n]\setminus I},\boldsymbol y)     \eta^{(n-|I|)} (\dd\boldsymbol y)  ,
		\end{multline}
		which is the multidimensional analogue of \eqref{eq: duality bound driv brownian gas} when there are two possible values for the reservoir intensity.
			\item[ii)] The multidimensional Brownian motion  satisfies \eqref{eq: cond for duality} with $m$ given by the Lebesgue measure (see, e.g. \cite[Theorem 4.4]{Bass}): thus the multidimensional boundary driven Brownian gas satisfies \eqref{eq: duality mult dim diff}.
			\item[iii)] In one dimension,  all diffusion processes satisfy \eqref{eq: cond for duality} (see, e.g. \cite[pag.13]{handbook}). In particular, consider the diffusion process on $\R$ with generator 
			\be
			\label{gen-diff}
			{\cal L}f(y)=\frac { 1} 2 \, \sigma^2(y)\, \frac{\dd^2 f}{\dd y^2}(y) +  b(y)\,\frac{\dd f}{\dd y}(y) 
			\ee
			where the drift $b$ and the diffusivity $\sigma$ are continuous functions and with $\sigma^2(x)\ge \delta>0$ for each $x\in (0,1)$. Then \eqref{eq: cond for duality} holds with $m$ given by 
			\be
			\label{rev-1d}
			m(\dd x)=\frac{1}{\sigma^2(x)} \;  \exp\left(2\int_{x_0}^x \frac{b(y)}{\sigma^2(y)}\, \dd y\right) \dd x\,,
			\ee
			for an arbitrary $x_0\in (0,1)$ (see, e.g. \cite[pag.17]{handbook}). 
			\item[iv)] We refer to \cite{k1978} for conditions on the coefficients $a_{i,j}$ and $b_i$ in \eqref{gen-diffus} ensuring that \eqref{eq: cond for duality} holds.
	\end{itemize}
\end{remark}

\bpr
Coherently with what we have done in  Section \ref{section: self inter and self dual without res},  we provide a proof relying on generating functions which uses the identity  \eqref{last}. Let $u:\mathfrak D\to \R$ bounded and measurable. By the independence of $\Theta_t$ and $\xi_t$ we have
\allowdisplaybreaks\beq
&&\E^{\lambda}_\eta\left[\exp\left(\int_\mathfrak D \log(1-u(z))\eta_t(\dd z)\right)\right]
\\
&=&
\E_{\mathcal P_{\lambda_t}}\left[\exp\left(\int_{\mathfrak D} \log(1-u(z))\Theta_t(\dd z)\right)\right] \E^{\rm{abs}}_{\eta}\left[\exp\left(\int_\mathfrak D \log(1-u(z))\xi_t{\big |}_{\mathfrak D} (\dd z)\right)\right],
\nonumber
\eeq
where $\E_{\mathcal P_{\lambda_t}}$ denotes the expectation in the Poisson point process $\Theta_t$.
Notice in particular that $\E_{\eta}\left[\exp\left(\int_\mathfrak D \log(1-u(z))\xi_t{\big |}_{\mathfrak D} (\dd z)\right)\right]=\E_{\eta}\left[\exp\left(\int_\mathfrak D \log(1-u(z))\xi_t (\dd z)\right)\right]$ and that for $\{\xi_t, t\ge 0\}$ Theorem \ref{floreanithm} applies.

Using \eqref{last} combined with \eqref{factmom} and \eqref{duality}, we obtain
\allowdisplaybreaks\beq\label{berbak}
&&\E^{\lambda}_\eta\left[\exp\left(\int \log(1-u(z))\eta_t(\dd z)\right)\right]
\\
&=&
\left(1+\sum_{n=1}^\infty\frac{(-1)^n}{n!}\int u^{\otimes n} (z_1, \ldots, z_n) \lambda_t^{\otimes n}(\dd(z_1\ldots z_n))\right)
\nonumber\\
&&\times \left(1+\sum_{n=1}^\infty\frac{(-1)^n}{n!}\int u^{\otimes n}(z_1,\ldots, z_n) (S^{\otimes n}_t)^*\eta^{(n)} (\dd(z_1, \ldots, z_n))\right)
\nonumber\\
&=&
1+
\sum_{k,l} \frac{(-1)^{k+l}}{k! \; l!}\int u^{\otimes (k+l)}(z_1, \ldots, z_{k+l})\, \lambda_t^{\otimes k}(\dd(z_1,\ldots, z_k))
\otimes (S^{\otimes l}_t)^*\eta^{(l)} (\dd(z_{k+1}, \ldots, z_{k+l})) \nonumber
\eeq
On the other hand we have
\begin{equation*}
\E^{\lambda}_\eta\left[\exp\left(\int \log(1-u(z))\eta_t(\dd z)\right) \right]= 1+\sum_{n=1}^\infty\frac{(-1)^n}{n!}\int u^{\otimes n}(z_1,\ldots, z_n)
\E^{\lambda}_\eta[\eta^{(n)}_t](\dd(z_1, \ldots, z_n)).
\end{equation*}
Then, via identification of the terms with $n$-fold tensor product of $u$ in the last expression in \eqref{berbak}
and the right hand side of the above identity, we obtain the following equality for all $n\in\N$:
\beq
&&\int u^{\otimes n} (z_1, \ldots, z_n) \E^{\lambda}_\eta[\eta^{(n)}_t] (\dd(z_1, \ldots, z_n ))
\nonumber\\
&=&
\sum_{k=0}^n {n\choose k}  \int u^{\otimes n} (z_1, \ldots, z_n) \lambda_t^{\otimes k}(\dd(z_1,\ldots, z_k))
\otimes (S^{\otimes (n-k)}_t)^*\eta^{(n-k)} (\dd(z_{k+1}, \ldots, z_{n}))
\eeq
Via the above mentioned density argument of linear combinations of $u^{\otimes n}$ this implies \eqref{eq: inter mult dim diff proc}.

Recalling the definition of $\lambda_t(\dd z)$ we have that 
\begin{align*}\lambda_t^{\otimes k} (\dd(z_1,\ldots, z_k)) &=  \left( \prod_{i=1}^{k} \int_{\partial  \mathfrak D} \lambda(u) \P^{\rm{abs}}_{z_i}( \tau_{\partial  \mathfrak D}\le t, X_{\tau_{ \partial  \mathfrak D}} \in\dd u )\right)\mu^{\otimes k}(\dd(z_1,\ldots, z_k))\\
&=\E^{\rm{abs}}_{\boldsymbol z_{[k]}}\left[e^{\int_{\partial\mathfrak D} \log (\lambda)\dd\xi_t}\boldsymbol 1_{\{\xi_t(\partial \mathfrak D)=k\}} \right]\mu^{\otimes k}(\dd(z_1,\ldots, z_k)).
\end{align*}

If now we assume that \eqref{eq: cond for duality} holds and choosing  $\mu = m$,
we obtain, integrating a bounded and permutation invariant function $f_n:\mathfrak  D^n\to\R$ 
\begin{align*}
&\E^{\lambda}_\eta\left[  \int_{\mathfrak  D^n} f_n (z_1,\ldots, z_n) \eta^{(n)}_t(\dd(z_1,\ldots, z_n))   \right]\\
&=\sum_{k=0}^n {n\choose k} \E^{\rm{abs}}_{\boldsymbol z_{[k]}}\left[e^{\int_{\partial\mathfrak D} \log (\lambda)\dd\xi_t}\boldsymbol 1_{\{\xi_t(\partial \mathfrak D)=k\}} \right]\int_{\mathfrak  D^n}f_n(\boldsymbol z) m^{\otimes k}(\dd\boldsymbol z_{[k]})
\otimes (S^{\otimes (n-k)}_t)^*\eta^{(n-k)} (\dd \boldsymbol z_{[n]\setminus [k]})\\
&=\sum_{k=0}^n {n\choose k} \E^{\rm{abs}}_{\boldsymbol z_{[k]}}\left[e^{\int_{\partial\mathfrak D} \log (\lambda)\dd\xi_t}\boldsymbol 1_{\{\xi_t(\partial \mathfrak D)=k\}} \right]\\
&\quad \times \int_{\mathfrak  D^n}\left(\int_{\mathfrak  D^{n-k}}f_n(\boldsymbol z_{[k]},\boldsymbol y)\mathfrak p_t^{(n-k)}(\boldsymbol z_{[n]\setminus [k]}, \boldsymbol y)m^{\otimes n-k}(\dd \boldsymbol y)\right) \; m^{\otimes k}(\dd\boldsymbol z_{[k]})
\otimes \eta^{(n-k)} (\dd \boldsymbol z_{[n]\setminus [k]})
\end{align*}
where we recall that $\mathfrak p^{(r)}_t((v_1,\ldots,v_r),(u_1,\ldots,u_r))=\prod_{i}^r\mathfrak p_t(v_i,u_i)$.
Exchanging the integrals and using the the symmetry of the functions $\mathfrak p_t(\cdot,\cdot)$ leads to 
\begin{align*}
&\E^{\lambda}_\eta\left[  \int_{\mathfrak  D^n} f_n (z_1,\ldots, z_n) \eta^{(n)}_t(\dd(z_1,\ldots, z_n))   \right]\\
&=\sum_{k=0}^n {n\choose k} \int_{\mathfrak  D^n} m^{\otimes k}(\dd\boldsymbol z_{[k]})\otimes m^{\otimes n-k}(\dd \boldsymbol y) f_n(\boldsymbol z_{[k]},\boldsymbol y) \\&\quad \times \left(  \E^{\rm{abs}}_{\boldsymbol z_{[k]}}\left[e^{\int_{\partial\mathfrak D} \log (\lambda)\dd\xi_t}\boldsymbol 1_{\{\xi_t(\partial \mathfrak D)=k\}} \right]\int_{\mathfrak  D^{n-k}}\mathfrak p_t^{(n-k)}( \boldsymbol y,\boldsymbol z_{[n]\setminus [k]}) \;    \eta^{(n-k)} (\dd \boldsymbol z_{[n]\setminus [k]})           \right)
\end{align*}
which, upon renaming the variables, can be rewritten as
\allowdisplaybreaks\begin{align*}\label{cat}
\E^{\lambda}_\eta\left[  \int_{\mathfrak  D^n} f_n \dd\eta^{(n)}_t   \right]
=&\int_{\mathfrak  D^n}  f_n(\boldsymbol z) \left(\sum_{k=0}^n {n\choose k}   \E^{\rm{abs}}_{\boldsymbol z_{[k]}}\left[e^{\int_{\partial\mathfrak D} \log (\lambda)\dd\xi_t}\boldsymbol 1_{\{\xi_t(\partial \mathfrak D)=k\}} \right]\right.\\&\left.\qquad \times \int_{\mathfrak  D^{n-k}}\mathfrak p_t^{(n-k)}(\boldsymbol z_{[n]\setminus [k]},\boldsymbol y)  \;   \eta^{(n-k)} (\dd\boldsymbol y)           \right)m^{\otimes n}(\dd\boldsymbol z).
\end{align*}
By taking the symmetrization of the above expression in brackets in the right hand side we obtain \eqref{eq: duality mult dim diff} and the proof is concluded .
\epr

\vskip.5cm

 We conclude this section by looking  at the evolution of a Poisson distributed particle cloud and by using duality to show the existence and the uniqueness of the stationary distribution for the system of boundary driven independent particles.
Let $\rho$ be a finite measure on $\mathfrak D$ 
and  denote by ${\cal P}_{\rho}$ the Poisson point configuration with intensity $\rho$.

\bt\label{theorem: stationary measure}
Let $\{\eta_t,\, t\ge 0\}$ be the boundary driven diffusion gas in the domain $\mathfrak D$  defined in 
\eqref{MDG} and let  $\mu(dx)$ be the finite measure on $\mathfrak D$ appearing in \eqref{eq: lambda t multidim BM}.\noindent
\begin{itemize}
\item[i)] If $\eta_0$ is distributed according to ${\cal P}_\rho$, then $\eta_t$ is the restriction to $\mathfrak D$ of the Poisson process on $\bar{\mathfrak  D}$ with intensity 
\be\label{pois}
\rho_t= S^*_t\rho + \lambda_t
\ee
with $\lambda_t$  defined in \eqref{intens}.
\item[ii)] Assume further \eqref{eq: condition for Markov} and take $\mu=m$. Then, the unique stationary measure for the boundary driven diffusion process is given by the distribution of a  Poisson point process with
intensity $$\lambda_{\infty}(\dd x)=h(x) m (\dd x)$$ where 
$$h(x)=\lambda(\infty, x)= \int_{\partial  \mathfrak D} \lambda(u) \P^{\rm{abs}}_{x}(  X_{\tau_{ \partial  \mathfrak D}} \in\dd u ).$$

Moreover, for any initial configuration $\eta$, the distribution of $\eta_t$ converges weakly as $t\to \infty$ to the distribution of the Poisson point process with intensity $\lambda_\infty(\dd x)=h(x) m(\dd x)$.
\end{itemize}
\et

\begin{remark}
Notice that, when $\lambda$ is a continuous function, $h:\mathfrak D\to \R$ given in Theorem \ref{theorem: stationary measure}(ii) is the solution of the following Dirichlet problem 
\begin{equation}
\begin{cases}
\mathcal L h=0 & \text{in }\mathfrak D\\
h=\lambda& \text{ on }\partial \mathfrak D
\end{cases}
\end{equation}
where $\mathcal L$ is the generator given in \eqref{gen-diff}.
In particular, for the one-dimensional boundary driven Brownian gas, $\mathcal L=\frac{1}{2}\frac{\dd^2}{\dd x^2}$ and 
$$\lambda_\infty(x)=\ll(1-x)+ \lr x.$$
\end{remark}
\vskip.5cm
\bpr
The evolution of ${\cal P}_{\rho}$ under independent copies of absorbed particles $\{X_t, t\ge 0\}$  is equal to $\Pi_{S^*_t \rho}$ by the Doob's theorem (Theorem \ref{doob}). Therefore \eqref{pois} follows from the fact that the independent
sum of two Poisson point processes is a Poisson point process with intensity measure the sum of the intensity measures.
Further,  notice that for every finite measure $\mu$ on $\mathfrak D$, we have
\[
(S_t^{\otimes n})^*\mu^{\otimes n}\longrightarrow 0
\]
as $t\to\infty$ because eventually all the mass from $\mu$ will be absorbed at the boundary $\partial \mathfrak D$.
Therefore, 
by taking the limit $t\to\infty$ in \eqref{eq: inter mult dim diff proc},  only the term $k=n$ survives and thus, the
$n$-th factorial moment measures converge to $\lambda_\infty^{\otimes n}(\dd \boldsymbol x)=\left(\prod_{i=1}^nh(x_i)\right) \mu^{\otimes n}(\dd \boldsymbol x)$ with $$h(x)=\lim_{t\to\infty}\lambda(t,x)=  \int_{\partial  \mathfrak D} \lambda(u) \P^{\rm{abs}}_{x}(  X_{\tau_{ \partial  \mathfrak D}} \in\dd u ).$$
This shows that the limiting distribution of $\eta_t$ is indeed Poisson with intensity measure $\lambda_{\infty}$. Since \eqref{eq: condition for Markov} implies that $\{\eta_t, t\ge 0\}$ is Markov, we conclude that the distribution of the Poisson point process with intensity $\lambda_{\infty}$ is the unique stationary measure.
\epr

\subsection{Boundary driven Markov gas }\label{section: general boundary driven gas}
In this section we provide another extension of the construction of Bertini and Posta \cite{bertini} for systems of particles that can make jumps and thus,  they do not necessarily hit the boundary when exiting a regular domain. Therefore, instead of associating a reservoir parameter function $\lambda$  to the boundary of the domain only, we need to associate it rather to the complement of the domain. We therefore consider  particles that evolve on a regular domain and are absorbed upon hitting a point in the complement of this domain.  The examples that we have in mind are jump Markov processes (see, e.g., \cite[Eq. 4]{kkoss}) with generator given by 

\be\label{FKD}
{\cal L}f(x)= \int_{\R^d}a(x-y) \left[f(y)-f(x)\right] \, \dd y \, ,
\ee
with $a(-x)=a(x)$ and $f:\R^d\to \R$ a Borel measurable function with compact support (a system of particles evolving accordingly to \eqref{FKD} is then called \textit{free Kawasaki dynamics})
and standard rotationally symmetric $\alpha$ stable processes (see, e.g., \cite{cmn2012}) with generator given by 
\begin{equation}\label{eq: gen alpha stab proc}
\mathcal L= \Delta^{\alpha/2}
\end{equation}
for $\alpha\in(0,2)$.

Let $\{Y_t, t\ge 0\}$ be a strong Markov process on $\R^d$. Let $\mathfrak D$ be regular domain of $\R^d$ (see Section \ref{section: boundary driven independent particles} ) and define $\mathfrak D^{\rm{ext}} :=\R^d\setminus \mathfrak D$.

Let $\{X_t, t\ge 0\}$ be the Markov process on $\R^d$ with semigroup $\{S_t, t\ge 0\}$ which evolves as $\{Y_t, t\ge 0\}$ on $\mathfrak D$ and is absorbed  upon hitting  $\mathfrak D^{\rm{ext}} $. We denote by  $\P^{\rm{abs}}_x$ the distribution of the trajectories of  $\{X_t, t\ge 0\}$ starting from $x\in \mathfrak D$,  by $\tau_{\mathfrak D^{\rm{ext}}}$ the hitting time of the set $\mathfrak D^{\rm{ext}}$. We assume  $\P^{\rm{abs}}_x(\tau_{\mathfrak D^{\rm{ext}}}<\infty)=1$.

  Let now $\lambda: \mathfrak D^{\rm{ext}}\to \R_+$ be a bounded measurable function giving the reservoir intensity at any $x\in\mathfrak D^{\rm{ext}}$ and let $\mu(dx)$ be a finite measure on $\mathfrak D$. We  then define the point configuration (on $\R^d$) valued process $\{\xi_t, t\ge 0\}$  arising from  independent copies  of the absorbed Markov  process $\{X_t, t\ge 0\}$ starting from $\xi_0=\sum_i \delta_{x_i}$, $x_i\in \mathfrak D$. I.e., for $\xi:=\sum_{i=1}^\NN\delta_{x_i}$, $x_i\in\mathfrak D$, we define 
  $$\begin{cases}\xi_t:=\sum_{i=1}^\NN\delta_{X_{t}(i)}\ ,&\\
  \xi_0=\xi\end{cases}$$
  where $\{X_t(i)\}_{t\ge 0}$ are independent copies of $\{X_t\}_{t\ge 0}$ such that $X_0(i)=x_i$ for any $i\in[\NN]$. For $\boldsymbol x=(x_1,\ldots,x_\NN)\in \mathfrak D^\NN$, we denote by $\E^{\rm{abs}}_{\boldsymbol x}$ the expectation in
  the process $\{\xi_t, t\ge 0\}$ starting from $\xi_0=\sum_{i=0}^N\delta_{x_i}$. 
    Finally we define $\Theta_t$, a Poisson point process on $ \mathfrak D$ independent of $\xi_t$ and with time dependent intensity $\lambda_t(\dd x)$ given by 
\begin{equation}\label{eq: lambda t multidim MP}
\lambda_t (\dd x)=\left(\int_{\mathfrak D^{\rm{ext}}} \lambda (y) \, \P^{\rm{abs}}_x( \tau_{\mathfrak D^{\rm{ext}} }\le t, X_{\tau_{\mathfrak D^{\rm{ext}} }} \in\dd y )\right) \mu(\dd x),
\end{equation}
which is supposed to be finite.
The \textit{boundary driven Markov  gas in the domain} $\mathfrak D$ \textit{with reservoir intensity} $\lambda$, denoted by $\{\eta_t, t\ge 0\},$ is then given, for any $t\ge 0$, by 
\be\label{MMG}
\eta_t=\xi_t {\big |}_{\mathfrak D}+\Theta_t,
\ee
viewed as a point configuration on $\mathfrak D$. We denote by $\E^\lambda_\eta$ the expectation in the process
defined via \eqref{MMG} initialized from $\eta$.

Also in this context, $\{\eta_t,\, t\ge 0\}$ is a  Markov process when the transition probability $p_t(x,\dd y)$ of the process $\{X_t, t\ge 0\}$ satisfy 
\be
p_t(x,\dd y) m(\dd x)=p_t(y,\dd x) m(\dd y) \qquad \text{on} \qquad  \mathfrak D \times \mathfrak D
\ee 
for a finite measure $m$ and for the reservoir  intensity in \eqref{eq: lambda t multidim BM} we choose $\mu=m$ (see   Theorem \ref{Markovianity} in the appendix).

 We then have the following result generalizing Theorem \ref{theorem: self-dual bound driv diff}. We omit the proof being a straightforward adaptation of the proof of  Theorem \ref{theorem: self-dual bound driv diff}. 
\begin{theorem}\label{theorem: inter a dual general gas}
		Let $\{\eta_t, t\ge 0\}$ be boundary driven Markov gas defined in \eqref{MMG}. Then for all $n\in\N$ and $t\ge 0$, it holds:
		\begin{itemize}
	\item[a)] for all bounded, measurable and permutation invariant  $f:\mathfrak D^n \to \R$ 
	\begin{equation}\label{eq: intertwining general bound driven}
	\E^\lambda_\eta\left[\int_{\mathfrak D^n} f(\boldsymbol z)\eta^{(n)}_t (\dd \boldsymbol z)  \right] = \sum_{k=0}^n {n\choose k} \int_{\mathfrak D^n} f(\boldsymbol z) \lambda_t^{\otimes k}(\dd\boldsymbol z_{[k]})
	\otimes (S^{\otimes n-k}_t)^*\eta^{(n-k)} (\dd \boldsymbol z_{[n]\setminus [k]});
	\end{equation}
	\item[b)] assume further that \begin{equation}\label{eq: cond dualk 2}
	p_t(x,\dd y)=\mathfrak p_t(x,y) m(\dd y)
	\end{equation} for a symmetric function $\mathfrak p_t(x,y)$ and a finite measure $m$ on $\mathfrak D$. Then,	choosing $\mu=m$, the following holds
	
	\begin{align}
	\frac{\dd\E^{\lambda}_\eta[\eta_t^{(n)} ]}{\dd m^{\otimes n }} (\boldsymbol z)=\sum_{I\subset [n]} \E^{\rm{abs}}_{\boldsymbol z_{I}}\left[e^{\int_{\mathfrak D^{\rm{ext}}} \log (\lambda)\dd\xi_t}\boldsymbol 1_{\{\xi_t(\mathfrak D^{\rm{ext}})=|I|\}} \right]	
	\int_{\mathfrak D^{n-|I|}}\mathfrak p_t^{(n-|I|)}(\boldsymbol z_{[n]\setminus I},\boldsymbol y)     \eta^{(n-|I|)} (\dd\boldsymbol y) 
	\end{align}
	\end{itemize}
\end{theorem}
\begin{remark}[Examples]
	\begin{itemize}
	\item[i)] The process $\{Y_t, t\ge 0\}$ with generator given in \eqref{FKD} is  reversible with respect to the Lebesgue measure (see, e.g., \cite[Remark 2.7]{kkoss}) but \eqref{eq: cond dualk 2} is not satisfied since each particle has a positive probability to stay in the initial position during any time interval $[0,t]$.
	\item[ii)]   A spherically symmetric  $\alpha$-stable processes on $\R^d$ with generator given in \eqref{eq: gen alpha stab proc} is strongly reversible w.r. to the Lebesgue measure (see, e.g. \cite[Eq. 4.4]{cmn2012}) and \eqref{eq: cond dualk 2} is fulfilled.
	\end{itemize}
\end{remark}
\section{The discrete case}
In this section we consider the discrete analogue of the boundary driven Brownian gas. Here by ``discrete'' we mean that the space on which the particles evolve is a lattice and the independent Brownians are replaced  by independent random walks.
Our first aim will be to show that such a process is equal   (in distribution) to the {\em reservoirs  process}  $\{\zeta_t, \, t\ge 0\}$ defined via the generator in \eqref{eq: IRW with Res CGGR}. We will then show how the boundary driven Brownian gas arises as a scaling limit of $\{\zeta_t, \, t\ge 0\}$.
\subsection{On the equivalence of two definitions of boundary driven independent random walks}
 We  consider the boundary driven Markov gas as explained  in Section \ref{section: general boundary driven gas}, where the process $\{Y_t, t\ge 0\}$ is chosen to be  the rate $\frac{1}{2}$ symmetric nearest neighbor random walk jumping on the integers and domain $\mathfrak D=V_N=\{1, \ldots, N\}$ with boundary $\{0,N+1\}$. The restriction to the nearest neighbor case is for simplicity only. The generalization to independent walkers with generic jump rates $c(x,y)$, $x,y\in \Z$, absorbed upon leaving $V_N$ is straightforward and so is the extension to more general graphs.
Let  $\widetilde V_N:=\{0, \ldots, N+1\}=V_N \cup \{0, N+1\}$ and 
 $\{X_t, t\ge 0\}$ be the process evolving as $\{Y_t, t \ge 0\}$ on $V_N$ and absorbed when hitting $0$ or $N+1$. Notice that in this context of nearest neighbor random walks, $\mathfrak D^{\rm{ext}}$ reduces to $\{0, N+1\}$. 
We start the process from an initial configuration $\eta\in \N^{ V_N}$,  viewed as a point configuration on $ V_N$, i.e. $\eta=\sum_i \delta_{x_i}$, where $x_i\in  V_N$ are the initial positions of the particles.  We define its time evolution as follows:
\be\label{bikost}
\eta_t= \xi_t{\big |}_{V_N}+ \Theta_t.
\ee
Here  $\xi_t$ is  the point configuration on $\tilde V_N$ at time $t$ arising from $\xi_0=\eta$ when all the particles in $\eta$  evolve as  independent copies of the process $X_t$ defined above. For $\boldsymbol z=(z_1,\ldots,z_n)\in V_N^n$, we denote by $\E^{\rm{abs}}_{\boldsymbol z}$ the expectation in the process $\{\xi_t, t\ge 0\}$ started from $\sum_{i=1}^n \delta_{z_i}$.
Further, for $\lambda=(\lambda_L,\lambda_R)\in\R^2_+$, $\Theta_t$ is a Poisson point process on $ V_N$  with intensity  defined by
\be\label{bolanko}
\lambda_t(\dd x)=\left( \ll\pee^{\rm{RW}}_x (X_t=0) + \lr \pee^{\rm{RW}}_x (X_t=N+1) \right) m(\dd x)
\ee
with $m(\dd x)$ denoting the counting measure and $\pee^{\rm{RW}}_x$ the path-space measure of the absorbed random walk $\{X_t, t\ge 0\}$. Thus,  $\{\Theta_t(\{x\}), x\in  V_N\}$ are independent random variables which are Poisson distributed with
parameter $\lambda_t(x)$. The process defined in \eqref{bikost} is the discrete analogue of the process defined in \eqref{ka} and from Theorem \eqref{theorem: inter a dual general gas} applied to this context we have that 
\be\label{eq: duality bound driv RW}
\frac{\dd\E^{\lambda}_\eta[\eta_t^{(n)} ]}{\dd m^{\otimes n}} (\boldsymbol z)=\sum_{I\subset [n]}  \E^{\rm{abs}}_{\boldsymbol z_{I}}\left[ \ll^{\xi_t(\{0\})}\lr^{\xi_t(\{1\})}\boldsymbol 1_{\{\xi_t(\{0,N+1\})=|I|\}} \right] \int_{V_N^{n-|I|}}\mathfrak p_t^{(n-|I|)}(\boldsymbol z_{[n]\setminus I},\boldsymbol y)     \eta^{(n-|I|)} (\dd\boldsymbol y) ,      
\ee
where $\mathfrak p_t^{(n)}(\boldsymbol z,\boldsymbol y)=\prod_{i=1}^n\pee^{\rm{RW}}_{z_i}(X_t=y)$ and $\E^{\lambda}_\eta$ denotes the expectation in the process $\{\eta_t, t\ge 0\}$ starting from $\eta$.

\vskip.3cm

Let us now compare the process $\eta_t$ with the process $\zeta_t$, the reservoir process with parameters $\ll, \lr$ introduced in Section \ref{section: recap duality}.

\bt\label{berta}
Let $\eta\in \N^{V_N}$.
Then $\{\zeta_t, t\ge 0\}$, denoting the reservoir process with parameters $\ll, \lr$ and  generator given  in \eqref{eq: IRW with Res CGGR} started from $\eta$, and
$\{\eta_t, t\ge 0\}$, denoting the boundary driven Markov gas defined in \eqref{bikost} started from $\eta$,  are equal in distribution.
\et
Notice that in the statement of the Theorem we are implicitly identifying the point configuration $\eta_t$ with the vector $(\eta_t(\{x\}))_{x\in V_N}$ of occupation variables.

\bpr
In order to prove the result we will make use of the duality relations \eqref{reservoirdual}  and \eqref{eq: duality bound driv RW}.

Indeed,  it suffices to show that for all $\xi=\sum_{i=1}^n\delta_{z_i}$, $z_i\in \N^{V_N}$  one has for all $\eta$ and $t\ge 0$
\be\label{bifi}
\E_\eta^{\rm{res}} \left[    \prod_{x=1}^N d(\xi(\{x\}), \zeta_t(x))\right]= \E^{\lambda}_{{\eta}}\left[   \prod_{x=1}^N d(\xi(\{x\}), \eta_t(\{x\}))\right].
\ee
By \eqref{reservoirdualfunction} and \eqref{reservoirdual} we have
\begin{align*}
\E_\eta^{\rm{res}} \left[     \prod_{x=1}^N d(\xi(\{x\}), \zeta_t(x))\right] =\E^{\rm{abs}}_{{\xi}} \left[  \ll^{\xi_t(\{0\})} \lr^{\xi_{t}(\{N+1\})} \prod_{x=1}^N d\left(\xi_t(\{x\}), \zeta(x)\right) \right].
\end{align*}
On the other hand, by \eqref{eq: relation fact meas clas dual}, we have
$$\E^{\lambda}_{{\eta}}\left[   \prod_{x=1}^N d(\xi(\{x\}), \eta_t(\{x\}))\right]=\E^{\lambda}_{\eta}[\eta^{(n)}_t(\{(z_1,\ldots,z_n)\}) ]$$

and by \eqref{eq: duality bound driv RW}
\begin{align}\label{discfactdens}
&\E^\lambda_\eta\left[ \eta^{(n)}_t(\{z_1,\ldots, z_n\})  \right]\\&=\sum_{I\subset [n]}  \E^{\rm{abs}}_{\boldsymbol z_{I}}\left[ \ll^{\xi_t(\{0\})}\lr^{\xi_t(\{1\})}\boldsymbol 1_{\{\xi_t(\{0,N+1\})=|I|\}} \right] \int_{V_N^{n-|I|}}\mathfrak p_t^{(n-|I|)}(\boldsymbol z_{[n]\setminus I},\boldsymbol y)     \eta^{(n-|I|)} (\dd\boldsymbol y) .
\end{align}
It thus remains to show that 
\begin{multline*}\sum_{I\subset [n]}  \E^{\rm{abs}}_{\boldsymbol z_{I}}\left[ \ll^{\xi_t(\{0\})}\lr^{\xi_t(\{1\})}\boldsymbol 1_{\{\xi_t(\{0,N+1\})=|I|\}} \right] \int_{V_N^{n-|I|}}\mathfrak p_t^{(n-|I|)}(\boldsymbol z_{[n]\setminus I},\boldsymbol y)     \eta^{(n-|I|)} (\dd\boldsymbol y) \\= \E^{\rm{abs}}_{\xi}\left[      \ll^{\xi_t(\{0\})}\lr^{\xi_t(\{N+1\})} \prod_{x \in V_N} d\left(\xi_t(\{x\}), \eta(\{x\})\right)  \right].
\end{multline*}
Notice that, for any $I\subset  [n]$, by \eqref{eq: relation fact meas clas dual}, we have
\begin{align}\label{eq: from pt to dual func}
&\int_{ V_N^{n-|I|}}\mathfrak p_t^{(n-|I|)}(\boldsymbol z_{[n]\setminus I},\boldsymbol y)     \eta^{(n-|I|)} (\dd\boldsymbol y)=\E^{\rm{abs}}_{\boldsymbol z_{[n]\setminus I}}\left[  \prod_{x=1}^N d(\xi_t(\{x\}), \eta(\{x\})) \right]
\end{align}
 We thus have,
\begin{align*}
&\E^\lambda_\eta\left[ \eta^{(n)}_t(\{z_1,\ldots, z_n\})  \right]\\&=\sum_{I\subset [n]}   \E^{\rm{abs}}_{\boldsymbol z_{I}}\left[ \ll^{\xi_t(\{0\})}\lr^{\xi_t(\{1\})}\boldsymbol 1_{\{\xi_t(\{0,N+1\})=|I|\}} \right] \E^{\rm{abs}}_{\boldsymbol z_{[n]\setminus I}}\left[  \prod_{x=1}^N d(\xi_t(\{x\}), \eta(\{x\})) \right]\\
&= \E^{\rm{abs}}_{\boldsymbol z}\left[\ll^{\xi_t(\{0\})}\lr^{\xi_t(\{N+1\})}   \prod_{x=1}^N d(\xi_t(\{x\}), \eta(\{x\}))\right],
\end{align*}
 where we used \eqref{discfactdens} and \eqref{eq: from pt to dual func} in the first equality and the independence of particles  in the second equality.
\epr

\subsection{Scaling limit}\label{section: scaling limit}
In this section we show how the process of independent random walkers with reservoirs $\ll$ and $\lr$, when appropriately rescaled in space
and time, and with rescaling of the reservoirs intensities, converges to the boundary driven Brownian gas.
We start with the following lemma.
\bl\label{Poissonlem}
Consider $\{\Theta^{(N)}\}_{N\ge1}$ a sequence of Poisson point processes on $(0,1)$ with the intensity measures
\be\label{bingo}
\lambda^{(N)}(\dd x)=\left(\frac1N\sum_{i=1}^N a_N(\tfrac i N)\delta_{i/N}  \right)(\dd x)
\ee
with $a_N:\{\tfrac{1}{N},\ldots,\tfrac{N-1}{N},1\}\to \R_+$.
Assume furthermore that whenever $i/N\to x\in [0,1]$ then also
\be\label{conva}
a_N( \tfrac i N)\to \alpha(x)
\ee
where $\alpha: [0,1]\to\R$ is a smooth function.
Then as $N\to\infty$, $\Theta^{(N)}$ converges to the Poisson point process $\Theta$ with intensity $\alpha(x) \dd x$.
\el
\bpr
Because sequences of Poisson point processes converge when the sequences of their intensity measures converge, it suffices to prove that \eqref{bingo} converges weakly to $\alpha(x) \dd x$ as $N\to \infty$. Let $f:[0,1]\to\R$ continuous, then
\[
\int f(x) \lambda^{(N)}(\dd x)=\frac1N \sum_{i=1}^N f(\tfrac{i}{N}) a_N(\tfrac iN)
\]
By the condition on $a_N(\tfrac iN)$, this sum converges to the Riemann integral $\int_0^1 f(x) \alpha(x) dx$.
\epr
\vskip.3cm
\noindent
We then have the following result. 
\bt
Consider the reservoir process $\{\zeta_{N,t},t\geq 0\}$ on the chain $\{1, \ldots, N\}$,  with reservoirs parameters $\frac{\ll}{N},\frac{\lr}{N}$ and generator given in \eqref{eq: IRW with Res CGGR}.
Define $\eta_{N,t} (\dd x)$ the point configuration on $[0,1]$ via
\be\label{rescaledeta}
{\cal Z}_{N,t}(\dd x)= \left(\sum_{i=1}^N \zeta_{N,tN^2}(i) \delta_{i/N}\right)(\dd x)
\ee
Assume that  at time $t=0$,
\be
{\cal Z}_{N,0}= \sum_{i=1}^\NN\delta_{x^{(N)}_i/N}
\ee
where $x^{(N)}_i/N\to x_i\in (0,1)$ for all $i=1,\ldots, \NN$.

Then as $N\to\infty$ the process $\{{\cal Z}_{N,t}( \dd x), t\geq 0\}$ converges (in the sense of convergence of finite dimensional distributions) to the boundary driven  Brownian gas
with  parameters $\ll, \lr$, started at the configuration $\sum_{i=1}^\NN\delta_{x_i}$.
\et
\bpr
As a consequence of Theorem \ref{berta},   the reservoir process $\zeta_{N,t}$ equals (in distribution) the boundary driven Markov gas $\eta_{N,t}$  obtained as a sum of
the configuration arising from letting the particles initially in the system evolve according to independent
random walkers absorbed at $0$ and $N+1$, and adding an independent Poisson point process  on $ V_N$ with intensity
\be\label{addpois}
\lambda_t(i)= \tfrac \ll N\; \pee^{\rm{RW}}_{i} (X_{t}=0) + \tfrac \lr N \; \pee^{\rm{RW}}_{i} (X_{t}=N+1)
\ee
where $\{X_t, t\ge 0\}$ denotes the random walk on $ V_N$ absorbed at the boundary $\{0, N+1\}$. Therefore after diffusive rescaling of space and time, the intensity of the  Poisson point process on $(0,1)$ modelling the reservoirs effect becomes 
\be\label{readpois}
\lambda^{(N)}_{t}( \dd x)= \sum_{i=1}^N \left(\tfrac{\ll}{N} \;\pee^{\rm{RW}}_{i} (X_{tN^2}=0) + \tfrac{\lr}{N}\;\pee^{\rm{RW}}_{i} (X_{tN^2}=N+1)\right)\delta_{i/N}(dx).
\ee
Because the absorbed random walk $X_{tN^2}/N$ converges weakly, as $N\to \infty$, to the  Brownian motion on $[0,1]$ absorbed at the boundaries, we can apply Lemma \ref{Poissonlem}
with
\[
a_N(\tfrac  iN)= \ll \,\pee^{\rm{RW}}_{i} (X_{tN^2}=0) + \lr\, \pee^{\rm{RW}}_{i} (X_{tN^2}=N+1)
\]
which converges, in the sense given by \eqref{conva}, to
\be\label{alph}
\alpha(x)= \ll \, \pee^{\rm{abs}}_x( \tau_L\leq t) + \lr \, \pee^{\rm{abs}}_x (\tau_R\leq t)
\ee
where $\tau_L, \tau_R$ are the hitting times of $0$, resp.\ $1$, and $\pee^{\rm{abs}}_x$ the path space measure of
Brownian motion started from $x$ and absorbed whenever hitting $0,1$.
Therefore, the  Poisson point processes \eqref{readpois} converge to the Poisson point processes with intensity \eqref{alph}.
Clearly, by the weak convergence of the absorbed random walk $X_{tN^2}/N$ to the absorbed Brownian motion, also the point configuration
corresponding to the time evolution of the independent walkers initially in the system converge to the point configuration arising by letting
the particles initially in the system evolve according to independent
absorbed Brownian motions. Because the evolution of the  particles initially in the system and the added Poisson process are independent, in the scaling limit, we obtain the sum of the evolution of the particles  initially in the system and
an independent Poisson point process with intensity \eqref{alph}, which is the boundary driven Brownian gas with reservoir
parameters $\ll, \lr$. 
\epr

\section{Orthogonal dualities}
In order to have a complete analogy with the   duality theory for independent random walks on a finite chain with reservoirs, we now investigate orthogonal dualities for the boundary driven Brownian gas.

\subsection{Known orthogonal dualities}

\paragraph{Closed discrete systems.}
Orthogonal self-duality functions are well known for the system of simple symmetric independent random walkers on $\Z^d$ described in Section \ref{section: recap duality} (see \cite{franceschini_stochastic_2017}, \cite{RS18}). More precisely, for any $\theta>0$, the factorized functions given by 
\begin{equation}\label{eq: ort dualities on Zd}
D_\theta^{\text{or}}(\xi,\eta)=\prod_{x\in \Z^d}C_{\xi(x)}(\eta(x), \theta)
\end{equation}
where $C_k(n,\theta)$ are the Charlier polynomials defined as
\be\label{Charlie}
C_k(n,\theta)=\sum_{\ell=0}^k\binom{k}{\ell}(-\theta)^{k-\ell}(n)_{\ell}
\ee
($(n)_\ell$ denotes the $\ell$-th falling factorial) are self-duality functions for the Markov process $\{\eta_t, \, t\ge 0\}$ with generator given in \eqref{eq: generator IRW}. The dualities in \eqref{eq: ort dualities on Zd} satisfy the following orthogonality relation w.r.t. the measure $\mu^{\rm{rev}}_\theta=\otimes_{x\in\Z^d}\rm{Poisson}(\theta)$ which is reversible for $\{\eta_t, \, t\ge 0\}$: for any $\xi,\, \xi'\in\N^{\Z^d}$
$$\int D_\theta^{\text{or}}(\xi,\eta)D^{\text{or}}(\xi',\eta) \dd \mu^{\rm{rev}}_{\theta}(\eta)=\boldsymbol 1_{\{\xi=\xi'\}}\frac{\xi!}{\theta^{|\xi|}}$$
where $\xi!:=\prod_{x\in\Z^d}\xi(x)!$ and $|\xi|=\sum_{x\in\Z^d}\xi(x)$.

Notice that the relation between orthogonal and classical dualities is   given by (see \cite[Remark 4.2]{FRSor})
\begin{align}
\label{eq: ort dual I}
D^{\rm{or}}_\theta(\xi, \eta)=\sum_{\xi'\le\xi}(-\theta)^{|\xi|-|\xi'|}\binom{\xi}{\xi'} D^{\rm{cl}}(\xi',\eta) =\sum_{I\subset [n]}(-\theta)^{n-|I|}D^{\rm{cl}}\left(\sum_{i\in I}\delta_{y_i},\eta\right),
\end{align}
where $\xi'\le \xi$ means that $\xi'(x)\le \xi(x)$ for any $x\in \Z^d$ and $\binom{\xi}{\xi'}:=\prod_{x\in \Z^d }\binom{\xi(x)}{\xi'(x)}.$
\vskip.4cm
\paragraph{Open discrete systems.}
Let us now reconsider the reservoir process with parameters $\ll,\,\lr$ defined in Section \ref{section: recap duality}.
In \cite{FRSor} the authors proved that the following functions, for $\theta>0$, 
\begin{align}\label{eq: ort dual fede}
D^{\rm{or}}_{\rm{res},\theta}(\xi,\zeta)
&=(\ll-\theta)^{\xi(0)} \Dor_\theta(\xi,\zeta)(\lr-\theta)^{\xi(1)}
\end{align}
with 
$$D_\theta^{\text{or}}(\xi,\zeta)=\prod_{x\in \widetilde V_N}C_{\xi(x)}(\zeta(x), \theta)$$
are duality functions between $\{\zeta_t, t\ge 0 \}$ the Markov process on $ V_N=\{1,\ldots,N\}$ with generator given in \eqref{eq: IRW with Res CGGR} and $\{\xi_t,\, t\ge 0\}$ the system of  random walkers
on $\widetilde V_N=\{0, \ldots, N+1\}$   absorbed at $\{0, N+1\}$. Notice that the orthogonality relation is w.r.t.  $\mu_\theta=\otimes_{x\in  V_N}\rm{Poisson}(\theta)$ which is not stationary for the reservoir process with general parameters $\ll,\,\lr$, but it is reversible for the reservoir process with parameters $\ll=\lr=\theta$, the last case referred as the reservoir process in \textit{equilibrium}.

\vskip.4cm
\paragraph{Closed systems in the continuum.}
Generalizations of orthogonal self-dualities for systems considered in Section \ref{section: self inter and self dual without res}, namely closed systems of independent Markov processes on  general Polish spaces $E$,  has been recently studied in \cite{floreani}.  More precisely, let $\eta_t=\sum_{i=1}^\NN \delta_{X_t(i)}$ with $\{X_t(i), \, t\ge 0\}$ independent copies of a  Markov process on $E$ started from  $x_i$, strongly reversible w.r.t. to a measure $m$. Then, the measure defined for any $t\ge 0$, $n\in \N$ and $\theta>0$ as
\begin{equation}
\eta_t^{(n),\theta}(\dd\boldsymbol z):= \sum_{I\subset [n]} (-\theta)^{n-|I|}\, \eta_t^{(|I|)}(\dd \boldsymbol z_I)\; m^{\otimes^{n-|I|}}(\dd\boldsymbol z_{[n]\setminus I})
\end{equation}
satisfies the following duality relation (see \cite[Corollary 4.2]{floreani})
$$ \frac{\dd\E^{\lambda}_\eta[\eta^{(n),\theta}_t] }{\dd m^{\otimes n}} (z_1,\ldots, z_n)
=
\int \prod_{i=1}^n \mathfrak p_t(z_i, y_i) \eta^{(n),\theta}(\dd(y_1, \ldots, y_n)) $$
and generalizes the orthogonal self-dualities given in \eqref{eq: ort dualities on Zd} in the following sense: 
\begin{itemize}
	\item[i)]  let $$\boldsymbol 1_{\boldsymbol B}(z_1,\ldots,z_n):=\left(\boldsymbol 1_{B_1}^{\otimes d_1}\otimes \ldots\otimes\boldsymbol 1_{B_K}^{d_K}\right)(z_1,\ldots,z_n)$$
for $\boldsymbol B=\{B_1,\ldots,B_K\}$ a family of mutually disjoint sets in $E$ and $\{d_1,\ldots, d_K\}$ such $\sum_{i=1}^K d_i=n$, then 
\begin{equation}\label{eq: relation eta or with ort pol}
\int \boldsymbol 1_{\boldsymbol B}(z_1,\ldots,z_n)\, \eta^{(n),\theta}(d(z_1,\ldots,z_n))=\prod_{\ell=1}^K  (-\theta m(B_\ell))^{d_\ell} \ C_{d_\ell}(\eta(B_\ell);\theta m(B_\ell))
\end{equation}
with $C_k(n,x)$ being the Charlier polynomials defined above.
\item[ii)] If we denote by  ${\cal P}_{\theta m}$ the distribution of a Poisson point process with intensity measure $\theta m$, then,  the following  orthogonal relation holds
\begin{align}\label{eq: ort rel eta or}
\E_{{\cal P}_{\theta m}} \left[\left(  \int f_n \dd \zeta^{(n),\theta}\right) \left(  \int g_{n'} \dd \zeta^{(n'),\theta}\right)\right] = \boldsymbol 1_{\{n=n'\}}\cdot  n! \int f_n \, g_n \, \dd (\theta m)^{\otimes n}
\end{align}
for $\zeta \sim{\cal P}_{\theta m}$ and  $f_n:E^n\to\R$, $g_{n'}:E^{n'}\to \R$ bounded and permutation invariant functions.
\end{itemize}
We refer to \cite{last} for a proof of the two above facts.
\vskip.3cm
The aim of the next section is to generalize the orthogonal dualities for the reservoir system given in \eqref{eq: ort dual fede} in the context of the boundary driven Brownian gas on $(0,1)$.

\subsection{Orthogonal dualities for the boundary driven Brownian gas}
Let us now consider the boundary driven Brownian gas on $(0,1)$ with parameters $\ll$ and $\lr$
$$\eta_t=\xi_t+\Theta_t$$
defined in Section \ref{bertsec}. We have previously proved that the factorial measure $\eta_t^{(n)}$ is the right object to study in order to have a duality result for boundary driven system. Inspired by the relation highlighted in the previous subsection between classical and orthogonal dualities we now  study  for any $n\in \N$ and $\theta>0$
\begin{equation}\label{eq: ort dual BDBG}
\eta_t^{(n),\theta}(\dd\boldsymbol z):= \sum_{I\subset [n]} (-\theta)^{n-|I|}\, \eta_t^{(|I|)}(\dd\boldsymbol z_I)\; m^{\otimes^{n-|I|}}(\dd\boldsymbol z_{[n]\setminus I}),
\end{equation}
viewed as a measure on $(0,1)^n$.
 Here $m(\dd z)$ is the Lebesgue measure on $(0,1)$ and the orthogonality properties \eqref{eq: relation eta or with ort pol} and \eqref{eq: ort rel eta or} hold for \eqref{eq: ort dual BDBG} for, respectively,  $\boldsymbol B=\{B_1,\ldots,B_K\}$ a family of mutually disjoint sets in $(0,1)$ with $\{d_1,\ldots, d_K\}$ such $\sum_{i=1}^K d_i=n$, and  bounded and permutation invariant functions $f_n:(0,1)^n\to\R$, $g_{n'}:(0,1)^{n'}\to \R$.

 Notice that the orthogonality relations holds true w.r.t. the intensity measure of the Poisson point process whose distribution is reversible for the  boundary driven Brownian gas \textit{in equilibrium}, namely with  $\ll=\lr=\theta$.

 Moreover, since we will integrate the above defined  measure $\eta^{(n),\theta}$ against $\mathfrak p_t^{(n)}(\cdot,\cdot):[0,1]^n\times[0,1]^n\to\R$, i.e. a function defined on $\{0,1\}$ as well, we extend $\eta^{(n),\theta}$ in the following way: we define $\bar m(d z)=m(dz)+\delta_0(dz)+\delta_1(dz)$ and we denote
 \begin{equation}\label{eq: extension eta or}
 \eta_t^{[n],\theta}(\dd{\boldsymbol z}):= \sum_{I\subset [n]} (-\theta)^{n-|I|}\, \eta_t^{(|I|)}(\dd{\boldsymbol z}_I)\; {\bar m}^{\otimes^{n-|I|}}(\dd {\boldsymbol z}_{[n]\setminus I})\end{equation}
  whenever integrated against functions being non zero also at the boundary $[0,1]$. Notice that $\int_{[0,1]}\mathfrak p_t(x,y)\bar m(\dd y)=1$ for any $x\in [0,1]$ and that we used the brackets $[\cdot]$ in the upper index of $\eta_t^{[n],\theta}$ to emphasize the difference with $\eta_t^{(n),\theta}$.

\vskip.3cm
We then have the following theorem, providing orthogonal dualities between the boundary driben Brownian gas  and the system of independent Brownian motions on $[0,1]$ absorbed at the boundaries.

\begin{theorem}\label{theorem: ort dual BDBG}
For the boundary driven Brownian gas, the expectation of the measure  given in \eqref{eq: ort dual BDBG} at time $t\ge 0$ is absolutely continuous w.r.t. $m^{\otimes n}$ with the following density:
\begin{multline}
 \frac{\dd\E^{\lambda}_\eta[\eta^{(n),\theta}_t] }{\dd m^{\otimes n}} (\boldsymbol z)=\sum_{J\subset [n]}\E^{\rm{abs}}_{\boldsymbol z_J}\left[  \ll^{\xi_t(\{0\})}\lr^{\xi_t(\{1\})}\boldsymbol 1_{\{\xi_t(\partial E)=|J|\}} \right]\int_{E^{n-|J|}} \mathfrak p_t^{(n-|J|)}(\boldsymbol z_{[n]\setminus J},\boldsymbol y)\eta^{[n-|J|],\theta}(\dd\boldsymbol y).
\end{multline}
\end{theorem}
\begin{proof}
	Using \eqref{eq: ort dual BDBG} and  \eqref{eq: duality bound driv brownian gas}  we have 
	\begin{align*}
	&\E^{\lambda}_\eta[\eta^{(n),\theta}_t](\dd\boldsymbol z)=\E^{\lambda}_\eta\left[ \sum_{I\subset [n]} (-\theta)^{n-|I|}\, \eta_t^{(|I|)}(\dd\boldsymbol z_I)\; m^{\otimes^{n-|I|}}(\dd\boldsymbol z_{[n]\setminus I})  \right]\\
	&=\sum_{I\subset [n]} (-\theta)^{n-|I|}\E_{\eta}\left[ \eta_t^{(|I|)}(d\boldsymbol z_I) \right]m^{\otimes^{n-|I|}}(\dd\boldsymbol z_{[n]\setminus I}) \\
	&=\sum_{I\subset [n]} (-\theta)^{n-|I|}\left(\sum_{J\subset I} \E^{\rm{abs}}_{\boldsymbol z_J}\left[  \ll^{\xi_t(\{0\})}\lr^{\xi_t(\{1\})}\boldsymbol 1_{\{\xi_t(\partial E)=|J|\}} \right]\right.\\
	&\left.\qquad \qquad \qquad \qquad \qquad \qquad \qquad \times\int_{E^{|I|-|J|}}\mathfrak p_t^{(|I|-|J|)}( \boldsymbol z_{I\setminus J},\boldsymbol y)\; \eta^{(|I|-|J|)} (\dd \boldsymbol y)\right)m^{\otimes^{n-|J|}}(\dd\boldsymbol z_{[n]\setminus J})\\
	\end{align*}
	and by exchanging the order of the summation in the last expression above we obtain
	\begin{align*}
	&\E^{\lambda}_\eta[\eta^{(n),\theta}_t](\dd\boldsymbol z)
	\\&=\sum_{J\subset [n]}\E^{\rm{abs}}_{\boldsymbol z_J}\left[  \ll^{\xi_t(\{0\})}\lr^{\xi_t(\{1\})}\boldsymbol 1_{\{\xi_t(\partial E)=|J|\}} \right] \\
	&\qquad \qquad \qquad \qquad \qquad \times\left(   \sum_{I\subset [n]\setminus J}(-\theta)^{n-|I|-|J|}  \int_{E^{|I|}}\mathfrak p_t^{(|I|)}( \boldsymbol z_{I},\boldsymbol y) \; \eta^{(|I|)} (\dd \boldsymbol y) \right)m^{\otimes^{n-|J|}}(\dd\boldsymbol z_{[n]\setminus J})
	\end{align*}
	We conclude by noticing that 
	\begin{multline}
	\sum_{I\subset [n]\setminus J}(-\theta)^{n-|I|-|J|} \int_{E^{|I|}}\mathfrak p_t^{(|I|)}( \boldsymbol z_{I},\boldsymbol y)\; \eta^{(|I|)} (\dd \boldsymbol y)=\int_{E^{n-|J|}} \mathfrak p^{(n-|J|)}(\boldsymbol z_{[n]\setminus J},\boldsymbol y)\; \eta^{[n-|J|],\theta}(\dd\boldsymbol y).
	\end{multline}
	which can be proved using \eqref{eq: extension eta or}.
\end{proof}
\vskip.3cm
\noindent	
Notice that the same result holds for any boundary driven system of strongly reversible Markov processes as the ones treated in Section \ref{section: boundary driven independent particles} and for the discrete system defined in \eqref{bikost}.

We thus conclude the section by showing that indeed Theorem \ref{theorem: ort dual BDBG} generalizes the duality relation for the discrete system  w.r.t. the orthogonal dualities given in \eqref{eq: ort dual fede}. Notice, indeed, that we have, from \eqref{eq: ort dual BDBG} and \eqref{eq: ort dual fede}, $$\E^{\lambda}_\eta\left[ \eta^{(n),\theta}_t(\{z_1,\ldots, z_n\})  \right]=\E^{\lambda}_{\eta}\left[   D^{\rm{or}}_{\theta}\left(\xi, \eta_t\right)  \right].$$ It thus remains to prove the following.

\begin{proposition}
Let $\eta_t$ denote the process defined in \eqref{bikost}. Then
for all $n\in \N$ and $z_1, \ldots, z_n\in  V_N$, denoting $ \sum_{i=1}^n \delta_{z_i}=\xi$, we have
\begin{align}
\sum_{J\subset [n]}\E^{\rm{abs}}_{\boldsymbol z_J}\left[  \ll^{\xi_t(\{0\})}\lr^{\xi_t(\{1\})}\boldsymbol 1_{\{\xi_t(\partial E)=|J|\}} \right]\int \mathfrak  p_t^{(n-|J|)}(\boldsymbol z_{[n]\setminus J},\boldsymbol y)\; \eta^{(n-|J|),\theta}(d\boldsymbol y)=\E^{\rm{abs}}_{\xi}\left[   D^{\rm{or}}_{\rm{res},\theta}\left(\xi_t, \eta\right)  \right].
\end{align}	
\end{proposition}
\begin{proof}
By the definition of $ \eta^{(n-|J|),\theta}$ we obtain
\begin{align}\label{eq: help get ort dual res}
&\nonumber\sum_{J\subset [n]}\E^{\rm{abs}}_{\boldsymbol z_J}\left[  \ll^{\xi_t(\{0\})}\lr^{\xi_t(\{1\})}\boldsymbol 1_{\{\xi_t(\partial E)=|J|\}} \right]\int \mathfrak p_t^{(n-|J|)}(\boldsymbol z_{[n]\setminus J},\boldsymbol y)\; \eta^{(n-|J|),\theta}(d\boldsymbol y)\\
&=\sum_{J\subset [n]}\E^{\rm{abs}}_{\boldsymbol z_J}\left[  \ll^{\xi_t(\{0\})}\lr^{\xi_t(\{1\})}\boldsymbol 1_{\{\xi_t(\partial E)=|J|\}} \right]\left(  \sum_{I\subset [n]\setminus J}  (-\theta)^{n-|J|-|I|}\int_{E^{|I|}} \mathfrak p_t^{(|I|)}(\boldsymbol z_{I},\boldsymbol y)\; \eta^{(|I|)}(d\boldsymbol y)  \right) \nonumber\\
&=\sum_{J\subset [n]}\E^{\rm{abs}}_{\boldsymbol z_J}\left[  \ll^{\xi_t(\{0\})}\lr^{\xi_t(\{1\})}\boldsymbol 1_{\{\xi_t(\partial E)=|J|\}} \right]\left(  \sum_{I\subset [n]\setminus J}  (-\theta)^{n-|J|-|I|}\; \E^{\rm{abs}}_{\boldsymbol z_{I}}\left[ \prod_{x=1}^N d(\xi_t(\{x\}), \eta(\{x\}))\right] \right)\nonumber\\
&=\sum_{J\subset [n]} \sum_{I\subset [n]\setminus J}  (-\theta)^{n-|J|-|I|} \; \E^{\rm{abs}}_{\boldsymbol z_J}\left[  \ll^{\xi_t(\{0\})}\lr^{\xi_t(\{1\})}\boldsymbol 1_{\{\xi_t(\partial E)=|J|\}} \right]   \E^{\rm{abs}}_{\boldsymbol z_{I}}\left[ \prod_{x=1}^N d(\xi_t(\{x\}), \eta(\{x\}))\right]\nonumber \\
&=\sum_{U\subset [n]} (-\theta)^{n-|U|} \; \E^{\rm{abs}}_{\boldsymbol z_U}\left[  \ll^{\xi_t(\{0\})}\lr^{\xi_t(\{1\})}\prod_{x=1}^N d(\xi_t(\{x\}), \eta(\{x\}))\right],
\end{align}	
where in the last line we used the independence of the particles. Combining \eqref{eq: help get ort dual res}, \eqref{eq: relation fact meas clas dual} and \eqref{reservoirdualfunction} we get 
\begin{align*}
&\sum_{J\subset [n]}\E^{\rm{abs}}_{\boldsymbol z_J}\left[  \ll^{\xi_t(\{0\})}\lr^{\xi_t(\{1\})}\boldsymbol 1_{\{\xi_t(\partial E)=|J|\}} \right]\int \mathfrak p_t^{(n-|J|)}(\boldsymbol z_{[n]\setminus J},\boldsymbol y)\; \eta^{(n-|J|),\theta}(d\boldsymbol y)\\
&=\sum_{U\subset [n]} (-\theta)^{n-|U|}\; \E^{\rm{abs}}_{\boldsymbol z_U}\left[  \ll^{\xi_t(\{0\})}\lr^{\xi_t(\{1\})} D^{\rm{cl}}(\xi_t,\eta)\right]=\sum_{U\subset [n]} (-\theta)^{n-|U|} \; \E^{\rm{abs}}_{\boldsymbol z_U}\left[   D^{\ll,\lr}(\xi_t,\eta)\right].
\end{align*}	
We then have, using again the independence of particles,
\begin{align*}
&\sum_{U\subset [n]} (-\theta)^{n-|U|}\; \E^{\rm{abs}}_{\boldsymbol z_U}\left[  \ll^{\xi_t(\{0\})}\lr^{\xi_t(\{1\})} D^{\rm{cl}}(\xi_t,\eta)\right]\\&=\E^{\rm{abs}}_{\boldsymbol z}\left[  \sum_{\xi'\le \xi_t} \binom{\xi_t}{\xi'}(-\theta)^{n-\xi'(\tilde V_N)}\; \ll^{\xi'(\{0\})} \, \lr^{\xi'(\{N+1\})} \, D^{\rm{cl}}\left(\xi'{\big |}_{V_N},\eta\right)\right]\\
&=\E^{\rm{abs}}_{\boldsymbol z}\left[      \left (   \sum_{\ell=0}^{\xi_t(\{1\})}\binom{\xi_t(\{1\})}{\ell}(-\theta)^{\xi_t(\{1\})-\ell} \, \ll^{\ell}  \right) \left( \sum_{\xi'\le \xi_t{ |}_{V_N}}\binom{\xi_t{\big |}_{V_N}}{\xi'}(-\theta)^{\xi_t(V_N)-\xi'(V_N)} \, \Dcl\left(\xi',\eta\right)  \right)   \right.\\
&\left. \qquad \qquad \qquad \qquad \qquad \qquad \qquad \qquad \qquad \  \qquad \qquad \times\left (   \sum_{r=0}^{\xi_t(\{N\})}\binom{\xi_t(\{N\})}{r}(-\theta)^{\xi_t(\{N\})-r} \, \lr^{r}  \right) \right]\\
&= \E^{\rm{abs}}_{\xi}\left[      (\ll-\theta)^{\xi_t(\{0\})}(\lr-\theta)^{\xi_t(\{N+1\})} D^{\rm{or}}_{\theta}\left(\xi_t{\big |}_{V_N}, \eta\right) \right]
\end{align*}
where the third identity follows from \eqref{eq: ort dual I}.  The  proof is concluded by noticing that 
$$(\ll-\theta)^{\xi_t(\{0\})}(\lr-\theta)^{\xi_t(\{N+1\})} D^{\rm{or}}_{\theta}\left(\xi_t{\big |}_{V_N}, \eta\right) =D^{\rm{or}}_{\rm{res},\theta}\left(\xi_t, \eta\right).$$

\end{proof}

\begin{appendices}

\section{Markov property of the boundary driven independent particles}\label{Markovianity}

\begin{theorem}\label{theorem: markov property MG}
	Assume that 
	\be\label{eq: reversibility p gas}
	p_t(x,\dd y) m(\dd x)=p_t(y,\dd x) m(\dd y) \qquad \text{on} \qquad \mathfrak D \times \mathfrak D \ 
	\ee  
	for some finite measure $m(\dd x)$ on $\mathfrak D$.
	Denote by $\caP_{\lambda_t}$ the law of a Poisson point process with intensity measure given in \eqref{eq: lambda t multidim MP} with the measure $m$ in place of $\mu$ and  let  $\Pres_t:\Omega\times \mathcal B(\Omega)\to [0,1]$, $t\ge 0$     defined by 
	\begin{equation}
	\Pres_t(\eta,B):=\int_{\Theta +\xi\in B}\caP_{\lambda_t}(\dd \Theta)P_t(\eta,\dd \xi),
	\end{equation} 
	where $P_t$ denotes the semigroup of the process $\{\xi_t, \, t\ge 0\}$.
	\noindent
	Then, the family $\Pres_t$, $t\ge 0$ is a time homogeneous transition function on $(\Omega, \mathcal B(\Omega))$ and there exists a Markov family with transition function $\Pres_t$.
\end{theorem}
\begin{proof}
We need to show that $\Pres_t$ satisfies the Chapman-Kolmogorv equation, which, due to \cite[Lemma A.3]{bertini}, boils down to check that for any continuous function $\psi$ with compact support strictly contained in $\mathfrak D$ and for any $s,t>0$
		$$\int \Pres_{s+t}(\eta,\dd \bar \eta )e^{i \int \psi \dd \bar \eta}=\int\int \Pres_s(\eta, \dd \zeta)\Pres_t(\zeta, \dd\bar \eta) e^{i \int \psi \dd \bar \eta}.$$
		By the definition of $\Pres_{t+s}$ and using \eqref{poislap}, we have that the left hand side is equal to $$\exp\left\{ \int (e^{i\psi}-1) \dd \lambda_{t+s}\right\}\int P_{s+t}(\eta, \dd \xi)e^{i\int \psi \dd \xi}.$$
		On the other hand, for the right hand side we have, 
		\begin{align*}
		&\int\int \Pres_s(\eta, \dd \zeta)\Pres_t(\zeta, \dd\bar \eta) e^{i \int \psi \dd \bar \eta}\\
		&=\int \caP_{\lambda_s}(\dd \Theta_1)\int P_s(\eta, \dd \xi_1) \int \caP_{\lambda_t}(\dd \Theta_2)\int P_t(\Theta_1+\xi_1,\dd\xi_2) e^{i\int \psi \dd(\xi_2+\Theta_2)}\\
		&=\exp\left\{ \int (e^{i\psi}-1) \dd \lambda_{t}\right\}\int \caP_{\lambda_s}(\dd \Theta_1)\int P_s(\eta, \dd \xi_1)\int P_t(\Theta_1+\xi_1,\dd\xi_2) e^{i\int \psi \dd\xi_2}\\
		&=\exp\left\{ \int (e^{i\psi}-1) \dd \lambda_{t}\right\}\\&\qquad\qquad\qquad\qquad\times\int \caP_{\lambda_s}(\dd \Theta_1)\int P_s(\eta, \dd \xi_1)\int P_t(\Theta_1,\dd\xi_{2,1}) e^{i\int \psi \dd\xi_{2,1}}\int P_t(\xi_1,\dd\xi_{2,2}) e^{i\int \psi \dd\xi_{2,2}}\\
		&=\exp\left\{ \int (e^{i\psi}-1) \dd \lambda_{t}\right\}\int \caP_{\lambda_s}(\dd \Theta_1)\int P_t(\Theta_1,\dd\xi_{2,1}) e^{i\int \psi \dd\xi_{2,1}}\int P_{t+s}(\eta,\dd\xi) e^{i\int \psi \dd\xi}\\
		\end{align*}
		where we used the definition of $\Pres_t$ first, the independence of the particles after and finally the Champan-Kolmogorov equation for $P_t$.
		Thus, it remains to show that 
		\begin{equation}\label{eq: intermediate step markov property multidim}
		\int \caP_{\lambda_s}(\dd \Theta_1)\int P_t(\Theta_1,\dd\xi_{2,1}) e^{i\int \psi \dd\xi_{2,1}}=\exp\left\{ \int (e^{i\psi}-1) \dd \lambda_{t+s}-\int (e^{i\psi}-1) \dd \lambda_{t}\right\}.
		\end{equation}
		By the independence of the particles and \eqref{poislap} follows that 
		$$\int \caP_{\lambda_s}(\dd \Theta_1)\int P_t(\Theta_1,\dd\xi_{2,1}) e^{i\int \psi \dd\xi_{2,1}}= \exp\left\{ \int S_t(e^{i\psi}-1)(x)\lambda_s(\dd x)  \right\},$$
		where $S_t$ denotes the semigroup of the absorbed Markov process upon hitting $\mathfrak D^{\rm{ext}}$ and which is given by 
		$$S_tf(x)=\int_{\mathfrak D}  p_t(x,\dd y)f(y)+ \int_{\mathfrak D^{\rm{ext}}} f(z)\P_x(\tau_{\mathfrak D^{\rm{ext}}}\le t,\, X_{\tau_{\mathfrak D^{\rm{ext}}}}\in \dd z)$$
		for any $f: \mathfrak D^*\to \R$ bounded function. Being $\psi$ zero at $\mathfrak D^{\rm{ext}}$ we have that 
		$$\int S_t(e^{i\psi}-1)(x)\lambda_s(\dd x) = \int \left( \int  p_t(x,\dd y)(e^{i\psi(y)}-1)  \right) \lambda_s(\dd x)$$
		and thus, \eqref{eq: intermediate step markov property multidim} is given if one proves that 
		\begin{equation}\label{eq: intermediate step markov property multidim22}
		\int  p_t(x,\dd y)\lambda_s(\dd x)=\lambda_{t+s}(\dd y)-\lambda_{t}(\dd y).
		\end{equation} 
		Using the  definition of $\lambda_s$ given in \eqref{eq: lambda t multidim MP}, we have that the left hand side of \eqref{eq: intermediate step markov property multidim22} is equal to 
		$$\left[\int  p_t(x,\dd y)\left(   \int_{ \mathfrak D^{\rm{ext}}} \lambda(z)\, \P_x(\tau_{ \mathfrak D^{\rm{ext}}}\le s,\, X_{\tau_{\mathfrak D^{\rm{ext}}}}\in \dd z) \right)  m(\dd x)  \right].$$
	Using the strong Markov property of the absorbed Markov  process, we have 
		\begin{align*}
		&\lambda_{t+s}(\dd y)=\left(   \int_{\mathfrak D^{\rm{ext}}} \lambda(z)\P_y(\tau_{\mathfrak D^{\rm{ext}}}\le t+s,\, X_{\tau_{\mathfrak D^{\rm{ext}}}}\in \dd z) \right)m(\dd y)\\
		&=\left(   \int_{\mathfrak D^{\rm{ext}}} \lambda(z)\P_y(\tau_{ \mathfrak D^{\rm{ext}}}\le t,\, X_{\tau_{ \mathfrak D^{\rm{ext}}}}\in \dd z) \right)m(\dd y)\\&\qquad \qquad \qquad \qquad + \left[\int  p_t(y,\dd x)\left(   \int_{\mathfrak D^{\rm{ext}}} \lambda(z)\P_x(\tau_{\mathfrak D^{\rm{ext}}}\le s,\, X_{\tau_{\mathfrak D^{\rm{ext}}}}\in \dd z) \right)  m(\dd y)  \right] \\
		&=\lambda_{t}(\dd y)+\left[\int  p_t(x,\dd y)\left(   \int_{\mathfrak D^{\rm{ext}}} \lambda(z)\P_x(\tau_{\mathfrak D^{\rm{ext}}}\le s,\, X_{\tau_{\mathfrak D^{\rm{ext}}}}\in \dd z) \right)  m(\dd x)  \right]
		\end{align*}
		where in the last identity we used the  condition \eqref{eq: reversibility p gas}. Thus  \eqref{eq: intermediate step markov property multidim22} follows, concluding the proof.
\end{proof}
	
	\end{appendices}

\end{document}